\documentclass[11pt,reqno]{amsart}
\usepackage{amsmath}
\usepackage{amssymb}
\usepackage{amsthm}
\usepackage{graphicx}
\usepackage{graphics}
\usepackage{hyperref}
\usepackage{cite}
\usepackage{tikz-cd}
\usepackage{tikz}
\usepackage[mathscr]{euscript}

\theoremstyle{plain}
\newtheorem{Th}{Theorem}[section]
\newtheorem{Lemma}[Th]{Lemma}

\newtheorem{Prop}[Th]{Proposition}

\theoremstyle{definition}
\newtheorem{Def}[Th]{Definition}

\newtheorem{Rem}[Th]{Remark}
\newtheorem{?}[Th]{Problem}
\newtheorem{Ex}[Th]{Example}

\newcommand{\C}{\mathbb{C}}
\newcommand{\Z}{\mathbb{Z}}

\newcommand{\tr}{\text{tr}}

\DeclareMathOperator{\End}{End}

\DeclareMathOperator{\id}{id}
\DeclareMathOperator{\Hom}{Hom}
\newcommand{\imm}{\operatorname{Im}}
\newcommand{\ldb}{\{\!\!\{}
\newcommand{\rdb}{\}\!\!\}}
\DeclareSymbolFont{rsfs}{U}{rsfs}{m}{n}
\DeclareSymbolFontAlphabet{\mathscrsfs}{rsfs}

\newcommand{\arrowIn}{
	\tikz \draw[-stealth] (-1pt,0) -- (1pt,0);
}

\numberwithin{equation}{section}
\hypersetup{linktocpage}

\begin{document}
	\title{Baxterization for the dynamical Yang--Baxter equation}
	\author{Muze Ren}
	\address{Section de mathématiques, University of Geneva,
		rue du Conseil-Général 7-9
		1205 Geneva, Switzerland.}
	\email{muze.ren@unige.ch}
    \begin{abstract}
	The Baxterization process for the dynamical Yang--Baxter equation is studied. We introduce the local dynamical Hecke ,Temperley--Lieb and Birman--Murakami--Wenzl operators, then by inserting spectral parameters, from each representation of these operators, we get dynamical R matrix under some conditions. As applications, we reformulate trigonometric degeneration of elliptical quantum group representations and also get dynamical R matrices for critical ADE integrable lattice models. Through Baxterization, we construct some one dimensional integrable systems that are dynamical version of the Heisenberg spin chain. 
	\end{abstract}
    \maketitle	
	\tableofcontents
\section{Introduction}
In 1990, Jones summarized and proposed the Baxterization procedure in \cite{Jones1990} when he discussed the relation between Yang-Baxter equation \eqref{eq:sec1_YBE} and the braid relation \eqref{eq:braid} 
\begin{equation}\label{eq:sec1_YBE}
R_iR^{'}_{i+1}R^{''}_i=R^{''}_{i+1}R^{'}_iR_{i+1},
\end{equation}
the idea as he wrote:''take a knot invariant, turn it into a coherent sequence of braid group representations, and then Baxterize (if possible) by inserting a spectral parameter so that YBE is satisfied ....''. Jones then gave two universal Baxterization examples, suppose that $\sigma_i,i=1,\dots,n$ satisfies the braid relations,
\begin{subequations}\label{eq:braid}
\begin{align}
&\sigma_i\sigma_{i+1}\sigma_i=\sigma_{i+1}\sigma_i\sigma_{i+1}\\
&\sigma_i\sigma_j=\sigma_j\sigma_i,\quad |i-j|>1
\end{align}
\end{subequations}

\begin{enumerate}
	\item Hecke case, also see \cite{Jimbo1986}. If $\sigma_i+\sigma^{-1}_i=x1$ and we define $
	R_i(\lambda)=e^{\lambda}\sigma_i+e^{-\lambda}\sigma_i^{-1}$, then it satisfies the Yang--Baxter equation \eqref{eq:sec1_YBE}, with the relation $\lambda''=\lambda'-\lambda$.\\
	\item Birman--Murakami--Wenzl case studied in \cite{Birman1989} and \cite{Murakami1989}. Let $E_i=\frac{1}{x}(\sigma_i+\sigma^{-1}_i)-1$, assume they satisfy the relations
	\begin{subequations}
	\begin{align*}
	&E^2_i=(a+a^{-1}-x)^{-1}E_i\\
	&E_i\sigma^{\pm}_{i-1}E_i=a^{\pm 1}E_i,E_i\sigma^{\pm}_{i+1}=a^{\mp 1}E_i\\
	&E_i\sigma_{i\mp 1}\sigma_i=E_iE_{i\pm 1}
	\end{align*}
	\end{subequations}
then $R_i(\lambda)=(e^{\lambda}-1)k\sigma_i+x(k+k^{-1})1+(e^{-\lambda}-1)k^{-1}\sigma^{-1}_i$
will satisfy the \eqref{eq:sec1_YBE} with $\lambda''=\lambda'-\lambda$.
\end{enumerate}

We want to apply the similar ideas to study the dynamical Yang--Baxter equation and try to have more understanding of the construction of its solutions. The dynamical Yang--Baxter equation was initially considered by Gervais and Neveu in the study of Liouville theory \cite{Gervais1984}. And later Felder in \cite{Felder1994,FelderICMP1995} rediscovered this equation in the study of quantization of Knizhnik--Zamolodchikov--Bernard equation and discovered the theory of elliptic quantum group and elliptic R matrix. Later the equations were widely studied in different aspects and have many connections to other fields, see for example \cite{FelderVarchenko1996,FelderTarasovVarchenkoAMS1997,FelderTarasovVarchenko1999,EtingofVarchenko1998,EtingofVarchenkoCMP1998,EtingofVarchenko1999,Aganagic2021,CostelloWittenYamazakiI2018,CostelloWittenYamazakiII2018,StokmanReshetikhin2020}, see also the standard books \cite{EtingofSchiffmann1999,EtingofLatour2005,Konno2020}.

We consider the generalized Baxterization process for the dynamical Yang--Baxter equation \eqref{eq:sec1_dYBE} based on the groupoid graded vector spaces language developed in \cite{Felder2020}. The main difference to the classical situation is the appearance of the dynamical shifts $h^{(i)}$.

\begin{equation}\label{eq:sec1_dYBE}
	\begin{split}
		&\check{R}^{(23)}(z-w,ah^{(1)})\check{R}^{(12)}(z,a)\check{R}^{(23)}(w,ah^{(1)})\\
		&=\check{R}^{(12)}(w,a)\check{R}^{(23)}(z,ah^{(1)})\check{R}^{(12)}(z-w,a)
	\end{split}
\end{equation}

It is very natural to define the local dynamical Hecke operators, Temperley--Lieb and Birman--Murakami--Wenzl operators using dynamical notations. For example the local Temperley--Lieb operators $T(a)$ associated to a map $\bar{\kappa}:\rm{Ob}(\rm{Groupoid})\to \C_{\ne 0}$ is defined by equations on a groupoid graded vector spaces $V$
\begin{subequations}
	\begin{align}
		&T(a)T(a)=\bar{\kappa}(a) T(a)\\
		&T^{(12)}(a)T^{(23)}(ah^{(1)})T^{(12)}(a)=T^{(12)}(a)\\
		&T^{(23)}(ah^{(1)})T^{(12)}(a)T^{(23)}(ah^{(1)})=T^{(23)}(ah^{(1)})
	\end{align}
\end{subequations}

We provides four kinds of examples of representations for the local dynamical Temperley--Lieb operators
\begin{enumerate}
	\item The representations \ref{eq:intro_T} related to the classification of coxeter diagrams, see the work of Goodman--Harpe--Jones, chapter 1 of the book \cite{Goodman-Harpe-Jones} and restricted quantum group \cite{Felder2020}. Here $S_a,S_b,S_c,S_d$ are entries of the Perro--Frobenious eigenvector of the corresponding diagram $\Gamma$, the groupoid structure is got from these diagrams.
	\begin{equation}\label{eq:intro_T}
		\begin{tikzpicture}[scale=1.5]
			\draw[->,dashed](-.6,.5) node[left]{ $T_{\Gamma}(d)=\oplus_{a,c}T_{\Gamma}\Big(\begin{matrix}d&c\\a&d\\\end{matrix}\Big)=\oplus_{a,c}\sqrt{\frac{S_aS_c}{S_dS_d}}$}
			-- (1.7,.5);
			\draw[->,dashed](.5,-.6) 
			-- (.5,1.6);
			\draw[->](0,1) node[left]{$d$} -- (.5,1) node[above,fill=white]{$\gamma$} -- (1,1)
			node[right]{$c$};
			\draw[->](1,1) -- (1,.5) node[right , fill=white]{$\delta$} -- (1,0);
			\draw[->](0,1) -- (0,.5) node[left, fill=white]{$\alpha$} -- (0,0);
			\draw[->]node[left]{$a$}(0,0) -- (.5,0) node[below,fill=white]{$\beta$}-- (1,0)
			node[right]{$d$};
		\end{tikzpicture}
	\end{equation}
     
	\item The elliptic representation related to the theory of representaitons of elliptic quantum group \cite{Felder1994,FelderVarchenko1996,EtingofVarchenko1998,EtingofVarchenko1999,Etingof2015}. In rank 2 case, with $[a]:=\theta\big( a/(L+1),\tau\big)/\big(\theta^{'}(0,\tau)/(L+1)\big)$, $\theta$ is the first Jacobi theta function. $\bar{\kappa}^{\text{ell}}(a)=\frac{[a+1]+[a-1]}{[a]}$, the operators are written as \eqref{eq:intro_E}, the groupoid structure is the infinite type A groupoid in figure \ref{figure:infinite groupoid}.
	\begin{equation}\label{eq:intro_E}
	\begin{split}
	T^{\rm{ell}}_A(a)&=\frac{\sqrt{[a-1][a+1]}}{[a]}E_{21}\otimes E_{12}+\frac{\sqrt{[a+1][a-1]}}{[a]}E_{12}\otimes E_{21}
	\\
	+&\frac{[a+1]}{[a]}E_{11}\otimes E_{22}+\frac{[ a-1]}{[a]}E_{22}\otimes E_{11}
	\end{split}
	\end{equation}
    \item Let $\ldb z\rdb:=\sinh(\frac{\pi z}{L+1})$, the hyperbolic representation of the dynamical Temperley-Lieb associated to $\bar{\kappa}^{\text{hyb}}(a)=2\cosh(\frac{\pi}{L+1})$ is \eqref{eq:intro_h}, the groupoid structure is the infinite type $A$ groupoid  in figure \ref{figure:infinite groupoid}.
    
    \begin{equation}\label{eq:intro_h}
    	\begin{split}
    		T_A^{\text{hyp}}(a)&=\frac{\sqrt{\ldb a-1\rdb \ldb a+1\rdb}}{\ldb a\rdb}E_{21}\otimes E_{12}+\frac{\sqrt{\ldb a+1\rdb \ldb a-1\rdb}}{\ldb a\rdb}E_{12}\otimes E_{21}\\
    		&+\frac{\ldb a+1\rdb}{\ldb a\rdb}E_{11}\otimes E_{22}+\frac{\ldb a-1\rdb}{\ldb a\rdb}E_{22}\otimes E_{11}
    	\end{split}
    \end{equation} 
    \item The trigonometric representation related to the trigometric degenerations of elliptic quantum group  is \eqref{eq:intro_t}, in rank 2 case, with $\bar{\kappa}^{\text{tri}}=2\cos(\frac{\pi}{L+1})$ and $\langle a\rangle=\sin (\frac{\pi z}{L+1})$.
    \begin{equation}\label{eq:intro_t}
    \begin{split}
    T^{\text{tri}}_A(a)&=\frac{\sqrt{\langle a-1\rangle\langle a+1\rangle}}{\langle a\rangle}E_{21}\otimes E_{12}+\frac{\sqrt{\langle a+1\rangle\langle a-1\rangle}}{\langle a\rangle}E_{12}\otimes E_{21}
    \\
    +&\frac{\langle a+1\rangle}{\langle a\rangle}E_{11}\otimes E_{22}+\frac{\langle a-1\rangle}{\langle a\rangle}E_{22}\otimes E_{11}
    \end{split}
    \end{equation}
\end{enumerate}

And following Jones cases, by inserting the spectral parameters, we consider three cases of Baxterization for the dynamical Yang--Baxter equation, 
\begin{enumerate}
	\item local dynamical Hecke case. Suppose that we have invertible operators $\sigma(a)$ defined on source fibers of groupoid graded vector space that satisfy
	\begin{subequations}
		\begin{align}
			&\sigma^{(12)}(a)\sigma^{(23)}(ah^{(1)})\sigma^{(12)}(a)=\sigma^{(23)}(ah^1)\sigma^{(12)}(a)\sigma^{(23)}(ah^{(1)})\\
			&\sigma(a)+\sigma^{-1}(a)=f(a)\id
		\end{align}
	\end{subequations}
	\begin{Th}[same as the Theorem \ref{th:baxter_Hecke}]
		Suppose that $f(a)=f(b)$ if there exists an arrow $\alpha$ with $s(\alpha)=a,t(\alpha)=b$, then the operator defined by
		\begin{align*}
			\check{R}(z,a)=e^z\sigma(a)+e^{-z}\sigma^{-1}(a)
		\end{align*}
		satisfies the dynamical Yang--Baxter equation \eqref{eq:sec1_dYBE}.

	\end{Th}
	\item local dynamical Temperley--Lieb case. In this case, if we assume that 
	\begin{align}
		\check{R}(x,a)=\id+xT(a)
	\end{align}
	where $T(a)$ is local dynamical Temperley--Lieb operators associated to $\bar{\kappa}$ on $V$, then we will get
	
	\begin{Th}[also Theorem \ref{th:baxter_lTl}, see also \cite{PEARCE1990}]
		Suppose that $x=f(z),x'=f(z'),x''=f(z''),z''=z'-z$ satisfies the following equation
		\begin{align}
			x^{''}=\frac{x'-x}{1+\bar{\kappa}(ah^1)x+xx'},\quad x^{''}=\frac{x'-x}{1+\bar{\kappa}(a)x+xx'}
		\end{align}
		then the operators $\check{R}(x,a)=\id+xT(a)$ satisfies the dynamical Yang-Baxter equation \eqref{eq:sec1_dYBE}.
	\end{Th} 
\begin{table}[h]
	\begin{center}
		\begin{tabular}{|c|c|c|c|c|c|}
			\hline
			Cases &ellpitic &hyperbolic &trigonometric & ADE & affine ADE\\
			\hline
			Groupoid type&infinite &infinite  &infinite  & finite & finite \\
			\hline
			x & (no) & $\frac{\sinh z}{\sinh(\lambda-z)}$&$\frac{\sin (z)}{\sin (\lambda-z)}$&$\frac{\sin (z)}{\sin (\lambda-z)}$&$\frac{z}{1-z}$\\
			\hline
		\end{tabular}
	\end{center}
	\caption{Baxterization table}
\end{table}

\begin{Rem}
The elliptic case is not baxterizable in the sense of ansatz $R=1+xT$ of the theorem \ref{th:baxter_lTl}.
\end{Rem}

\item local dynamical Birman--Murakami--Wenzl case.

	\begin{Th}[same as Theorem \ref{th:Baxter_BMW}]
		Suppose that the operators $U(a)$ are the local dynamical Birman--Murakami--Wenzl operator associated with $\bar{q},\bar{\nu}$ on $V$, and suppose that $
		\bar{q}(a)=\bar{q}(b),\bar{\nu}(a)=\bar{\nu}(b),$ if there exists an arrow $\alpha\in \pi$ with $s(\alpha)=a,t(\alpha)=b$. Then we define
		\begin{align*}
			\check{R}(u,v)[a]:=U(a)+\frac{\bar{q}(a)-\bar{q}^{-1}(a)}{v/u-1}+\frac{\bar{q}(a)-\bar{q}^{-1}(a)}{1+\bar{\nu}^{-1}(a)\bar{q}(a)v/u}K(a),
		\end{align*}
		and it satisfies the following two parameters dynamical Yang--Baxter equation
		\begin{align*}
			&\check{R}^{(12)}(u_2,u_3)[a]\check{R}^{(23)}(u_1,u_3)[ah^{(1)}]\check{R}^{(12)}(u_1,u_2)[a]\\
			&=\check{R}^{(23)}(u_1,u_2)[ah^{(1)}]\check{R}^{(12)}(u_1,u_3)[a]\check{R}^{(23)}(u_2,u_3)[ah^{(1)}]
		\end{align*} 
	\end{Th}
\end{enumerate}

We provide several applications of the Baxterization procedure of the dynamical Yang--Baxter equation mentioned above.

\begin{enumerate}
	\item The first application is about the trigonometric dynamical $\check{R}(x)$ in the representation theory of dynamical YBE equation \cite{Felder1994,FelderVarchenko1996,EtingofVarchenko1998,EtingofVarchenko1999} and restricted quantum groups \cite{Felder2020}, we show that it can be seen as a Baxterization of representations of local dynamical Temperley--Lieb operators.\\
	\item The second application is to derive some new dynamical $\check{R}$ matrices of two kinds of face type two dimensional integrable lattice models, critical ADE models (also called Pasquier models \cite{Pasquier1987b}) and Temperley--Lieb interaction models introduced by Owczarek and Baxter \cite{Owczarek1987}. 
	
	Critical ADE models are a series of face type lattice models introduced and studied mainly by Pasquier (\cite{Pasquier1987,Pasquier1987b, Pasquier1988}), these models are interesting as they are related to the unitary A-series of minimal models considered by Belavin--Polyakov--Zamolodchikov \cite{Belavin1984} and Friedan--
	Qiu--Shenker \cite{Friedan1984}. Temperley--Lieb interaction models are also very interesting models which unifies (meaning that partition functions are the same with carefully chosen parameters and boundary conditions) a series of interesting models including critical ADE models, six-vertex model, self-dual Potts model, critical hard hexagons model, see the nice lecture note of Pearce \cite{PEARCE1990}.
	
	Interesting questions about these models are the "quantum group" picture behind these models \cite{Drinfeld1987,Faddeev1998a,Faddeev1988} and the possibility of application of algebraic Bethe ansatz \cite{Faddeev1988}. With the dynamical $\check{R}$ matrices derived here and the algebraic Bethe ansatz of face type restricted model developed in \cite{Felder2020}, these questions can be answered.
	  
	\item The third application of Baxterization is to get a family of one dimensional integrable systems. For each representation of local dynamical Temperley--Lieb algebra that can be Baxterized, we can define a Hamiltonian that commute with a family of transfer matrices. They are some dynamical version of Heisenberg spin chains. For the restricted type A case, they were initial considered by Bazhanov--Reshetikhin \cite{BAZHANOV1989}, where they also calculate the eigenvalues and eigenvectors. For the unrestricted type A case, see the work of Etingof--Kirillov \cite{Etingof1994},  Felder--Varchenko \cite{Felder1997} and Etingof--Vachenko \cite{Etingof2002}, where they studied the spin generalization of Ruijsenarrs models and MacDonald theory.
	
	And also there is a recent breakthrough in the study of long range spin chains, Klabbers and Lamers in \cite{Klabbers2023} introduced two new integrable systems which unifies Inozemtsev and partially isotropic Haldane–Shastry chains. And by taking the short range limit of their new integrable systems, they got some new dynamical spin chains, which is very similar to the Hamiltonian we get \eqref{eq:Hamitonian}, one difference is we use the periodic boundary condition, they have a twist. Another is that in \eqref{eq:Hamitonian}, we have conjugation of $M(0,\alpha)$ which is some translation operator, but the equation (18) of \cite{Klabbers2023} is conjugated by some terms which contains nontrivial $\check{R}$ matrix. 
	
\end{enumerate}

\subsection{Outline of the paper}
The note is organized as follows, in the beginning of \ref{sec:the_content}, we first recall the basic definition of $\pi$ graded vector space \cite{Felder2020} and operators on it. Then in subsection \ref{sub:YBEandOther} and \ref{sub:ldo}, we define dynamical Hecke, Temperley--Lieb and Birman--Murakami--Wenzl operators, its local versions and the finite type representation of local Temperley--Lieb operators. In subsection \ref{sub:relation}, we discuss the relation between the local dynamical operators and the usual operators. In subsection \ref{sub:Baxterization}, we discuss the Baxterization of the dynamical Yang--Baxter equation. In subsection \ref{sub:example_ellptic}, we provides the unrestricted examples of local Temperley--Lieb operators and its Baxterization. In subsection \ref{sub:ADE}, we consider the Baxterization of restricted cases examples including ADE type and affine ADE type. Then in subsection \ref{sub:transfer_matrix}, we discuss the transfer matrix and hamiltonian spin chains.

{\bf{Acknowledgment}}
The author would like to thank Anton Alekseev for encouragement and support, Giovanni Felder for mentioning the reference \cite{Pasquier1988a} and Rinat Kashaev, Rob Klabbers and Jules Lamers for interesting discussions. Research of the author is supported by the grant number 208235 of the Swiss National Science Foundation (SNSF) and by the NCCR SwissMAP of the SNSF.

\section{Main constructions}\label{sec:the_content}
A groupoid $\pi$ is a small category in which every morphism is invertible. We denote its object by $\text{Ob}(\pi)$, for any $a,b\in \text{Ob}(\pi)$, the set of morphisms from an object $a$ to $b$ is denoted by $\pi(a,b)$, we also call the morphisms as arrows, the composition of arrows $\gamma\in \pi(a,b)$ and $\eta\in \pi(a,b)$ is denoted by $\eta\circ \gamma\in \pi(a,c)$. The object of $\pi$ is identified with the identity arrows. By abuse of notation, the set of arrows of $\pi$ is also denoted by $\pi$, then we can denote the source and target maps by $s,t:\pi\to \text{Ob}(\pi)$. The source fibers are $s^{-1}(a)$ and target fibers are $t^{-1}(a)$, for $a\in \rm{Ob}(\pi)$.

\begin{Ex}\label{ex:graph}
For any unoriented graph $\mathcal{G}$, it can be seen as a groupoid $\pi(\mathcal{G})$ by taking the vertices as objects, for each unoriented edge $e\in \mathcal{G}$, we have corresponding two inverse direction arrows $\alpha_{e},\alpha^{-1}_{e}\in \pi(\mathcal{G})$ and arrows formed by their compositions.
\end{Ex}

\begin{Ex}[Action groupoid]
Let $A$ be a set with right group action $A\times G\to A$, the action groupoid $A\rtimes G$ has the set of objects $A$ and for each $a'=ag$, there is an arrow $a\xrightarrow{g}a'$, thus an arrow is described by a pair $(a,g)\in A\times G$. The source and target are $s(a,g)=a, t(a,g)=ag$ and the composition is $(a',g')\circ (a,g)=(a,gg')$,with~ $a'=ag$. The identity arrows are $(a,e),a\in A$ where $e$ is the identity element in the group and the inverse of $(a,g)$ is $(ag,g^{-1})$.
\end{Ex}

\begin{Def}
	Let $\pi$ be a groupoid with set of objects $A$, a $\pi$-graded vector space of finite type over a field $k$ is a collection $(V_{\alpha})_{\alpha\in \pi}$ of finite dimensional vector spaces indexed by the arrows of $\pi$ such that for each $a\in A$, there are finitely many arrows $\alpha$ with source or target $a$ and nonzero $V_{\alpha}$.
\end{Def}

With groupoids, we can define the groupoid graded vector spaces with certain finite conditions.

\begin{Def}
	Let $\pi$ be a groupoid, a $\pi$-graded vector space of finite type over a field $k$ is a collection $(V_{\alpha})_{\alpha\in \pi}$ of finite dimensional vector spaces indexed by the arrows of $\pi$, such that for each $a\in \rm{Ob}(\pi)$, there are only finitely many arrows $\alpha$ with source or target $a$ and nonzero $V_{\alpha}$.
\end{Def}

For $a\in \rm{Ob}(\pi)$, the $a$ source fibers of the vector space $V\in \pi$ is the direct sum of components with source $a$, that is  $\oplus_{\alpha\in s^{-1}(a)}V_{\alpha}$. And similar for the target fibers.

The $k$ vector space $\Hom(V,W)$ of two $\pi$ graded vector spaces consists of families $(f_{\alpha})_{\alpha\in \pi}$ of linear maps $f_{\alpha}:V_{\alpha}\to W_{\alpha}$, that is $f_{\alpha}\in \Hom_{k}(V,W)$, the compositions of maps are also defined component-wise. The category of $\pi$ groupoid graded vector spaces is a monoidal category, we denote by $\text{Vect}_k(\pi)$, with the tensor product defined
\begin{align*}
(V\otimes W)_{\gamma}=\oplus_{\beta\circ \alpha=\gamma}V_{\alpha}\otimes W_\beta,\quad V,W\in \text{Vect}_k(\pi)
\end{align*}

In the following example, we introduce a family of automorphism induced by the maps from $\rm{Ob}(\pi)$ to nonzero complex numbers.

\begin{Ex}\label{ex:q}
Let $\bar{q}:\rm{Ob}(\pi)\to \C_{\ne 0}$ be a map from the objects of the groupoid to the nonzero complex numbers. For any $V_1,\dots,V_N\in\rm{Vect}_k(\pi)$, the map $\bar{q}$ induces endormorphisms $q^{(i)}\in \End(\otimes_{l=1}^NV_l), i=0,\dots,N$. For convenience $q^{(0)}$ is also simply denoted by $q$.

For any $\alpha_1,\dots,\alpha_{N}$ such that $\alpha=\alpha_N\circ \dots\alpha_1$, the $\alpha$ component of $q^{(i)}$ restricted to $\otimes^N_{l=1}V_{\alpha_l}$ is defined by
\begin{subequations}
\begin{align}
	&q^{(0)}_{\alpha}|_{V_{\alpha_1}\otimes V_{\alpha_2}\dots V_{\alpha_N}}:=\bar{q}(s(\alpha_1))\id,\\
	&q^{(i)}_{\alpha}|_{V_{\alpha_1}\otimes V_{\alpha_2}\dots V_{\alpha_N}}:=\bar{q}(t(\alpha_i))\id, \quad 1\le i\le N
\end{align}
\end{subequations}
where $\id$ is the identity operator in $\End_{k}(\otimes^N_{l=1}V_{\alpha_l})$. It is easily seen that $q^{(i)}$ are invertible, its inverse are the endormoprhisms $(q^{i}){-1}$ induced by $\bar{q}^{-1}$ which are defined by
\begin{equation}
\bar{q}^{-1}(a):=\frac{1}{\bar{q}}(a)
\end{equation}

On the $a$ source fibers of vector space $\End(\otimes^N_{l=1}V_l)$, the operators $q^{(i)}$ can also be denoted by
\begin{align}
\bar{q}(a),\quad i=0;\quad \bar{q}(ah^{(i)}),\quad i\ge 1
\end{align}
here $h^{(i)}$ are the "dynamical" notations, intuitively $ah^{(i)}$ means the target vertices of $i$th edge of a chain of arrows that starts at $a$.
\end{Ex}

For the $k$ additive category $\rm{Vect}_{k}(\pi)$, we can also define the dual groupoid graded vector space and the resulting category is an abelian pivotal monoidal category, see section 2 of \cite{Felder2020}.

\subsection{Yang-Baxter operator and other local operators}\label{sub:YBEandOther}
Let $k=\C$, a Yang-Baxter operator on $V\in \text{Vect}_k(\pi)$ is a meromorphic function $z\to \check{R}(z)\in \End(V\otimes V)$ of the spectral parameter $z\in \C$ with values in the endormorphisms of $V\otimes V$, obeying the Yang-Baxter equation
\begin{align}\label{eq:YBE}
	\check{R}(z-w)^{(23)}\check{R}(z)^{(12)}\check{R}(w)^{(23)}=\check{R}(w)^{(12)}\check{R}(z)^{(23)}\check{R}(z-w)^{(12)}
\end{align}
more generally, we can write the equation as:
\begin{align}\label{eq:gYBE}
	\check{R}(x)^{(23)}\check{R}(x')^{(12)}\check{R}(x'')^{(23)}=\check{R}(x'')^{(12)}\check{R}(x')^{(23)}\check{R}(x)^{(12)}
\end{align}
in $\End(V\otimes V\otimes V)$ for all generic values of the spectral parameters $x,x',x''$ and certain inversion (also called unitary) relation. And here $x,x',x''$ are certain functions of the $z,z',z''$ and satisfies the conditions $z''=z'-z$.

The restriction of $\check{R}(x)$ to $V_{\alpha}\otimes V_{\beta}$ for composable arrows $\alpha,\beta$ has components in each direct summand of the decomposition
\begin{align*}
\check{R}(x)|_{V_{\alpha}\otimes V_{\beta}}=\oplus_{\gamma,\delta}\mathcal{W}(x;\alpha,\beta,\gamma,\delta)
\end{align*}
The sum is over $\gamma,\delta$ such that $\beta\circ \alpha=\delta\circ \gamma$ and $\mathcal{W}$ is the component
\begin{align*}
\mathcal{W}(x;\alpha,\beta,\gamma,\delta)\in \Hom_k(V_{\alpha}\otimes V_{\beta},V_{\gamma}\otimes V_{\delta})
\end{align*}

Similarly we can define the local Hecke operator, local Temperley--Lieb operator and local Birman--Murakami--Wenzl operator which are without the spectral parameter.

\begin{Def}
For a map $\bar{q}:\rm{Ob}(\pi)\to \C_{\ne 0}$, it induces a map $q\in \End(V\otimes V)$ as in \ref{ex:q}. The local Hecke operator $S\in \End(V\otimes V)$ associated to $q$  is defined by the following equations
\begin{subequations}
\begin{align}
	&(S-q)(S+q^{-1})=0\\
	&S^{(12)}S^{(23)}S^{(12)}=S^{(12)}S^{(23)}S^{(12)}
\end{align}
\end{subequations}
\end{Def}

\begin{Def}
For a map $\bar{\kappa}:\rm{Ob}(\pi)\to \C_{\ne 0}$, it induces a map $\kappa\in \End(V\otimes V)$. The local Temperley--Lieb operator $T$ in $\End(V\otimes V)$  associated to $\kappa$ is  defined by the equations
\begin{subequations}
\begin{align}
	&T^2=\kappa T\\
	&T^{(12)}T^{(23)}T^{(12)}=T^{(12)}\\
	&T^{(23)}T^{(12)}T^{(23)}=T^{(23)}
\end{align}
\end{subequations}
\end{Def}

\begin{Def}
For two maps $\bar{q},\bar{\nu}:\rm{Ob}(\pi)\to \C_{\ne 0}$, let $V\in \text{Vect}_k(\pi)$ and $U$ be an invertible operator in $\End(V\otimes V)$ and
\begin{equation}
K:=\id- (q-q^{-1})^{-1}(U-U^{-1}),
\end{equation}
$U$ is called local Birman--Murakami--Wenzl operator if it satisfies the following equations
\begin{subequations}
\begin{align}
	&U^{(12)}U^{(23)}U^{(12)}=U^{(23)}U^{(12)}U^{(23)}\\
	&KT=TK=\nu K\\
	&K^{(23)}(T^{\epsilon})^{(12)}K^{(23)}=(\nu^{-\epsilon})^{(1)}K^{(23)},\quad \epsilon=\pm 1\\
	&K^{(12)}(T^{\epsilon})^{(23)}K^{(12)}=(\nu^{-\epsilon})K^{(12)},\quad \epsilon=\pm 1
\end{align}
\end{subequations}

\end{Def}

\subsection{Local dynamical operators}\label{sub:ldo}
Let $V\in \text{Vect}_k(\pi)$, the dynamical Yang-Baxter operator $\check{R}(z,a)$ is defined on the source fibers $s^{-1}(a)$ of $V$, $a\in \text{Ob}(\pi)$,
\[
\check{R}(z,a)\in \oplus_{\alpha\in s^{-1}(a)}\End_{k}\big((V\otimes V)_{\alpha}\big),
\]
which satisfies the dynamical Yang-Baxter equation
\begin{equation}\label{eq:dYBE}
\begin{split}
&\check{R}^{(23)}(z-w,ah^{(1)})\check{R}^{(12)}(z,a)\check{R}^{(23)}(w,ah^{(1)})\\
&=\check{R}^{(12)}(w,a)\check{R}^{(23)}(z,ah^{(1)})\check{R}^{(12)}(z-w,a)
\end{split}
\end{equation}
here we use the "dynamical" notation with the placeholder $h^{(i)}$:
\begin{align*}
	\check{R}^{(23)}(ah^{(1)})(u\otimes v\otimes w)=u\otimes \check{R}(t(\alpha_1))(v\otimes w),\quad \text{if} \quad u\in V_{\alpha_1},s(\alpha_1)=a
\end{align*}

And more generally, we write the dynamical Yang--Baxter equation as 
\begin{equation}\label{eq:gdYBE}
\begin{split}
	&\check{R}^{(23)}(x,ah^{(1)})\check{R}^{(12)}(x',a)\check{R}^{(23)}(x'',ah^{(1)})\\
	&=\check{R}^{(12)}(x'',a)\check{R}^{(23)}(x',ah^{(1)})\check{R}^{(12)}(x,a)
\end{split}
\end{equation}
and here $x,x',x''$ are functions of the $z,z',z''$ and satisfies the conditions $z''=z'-z$.

In the case of an action groupoid $\pi=A\rtimes G$ and its subgroupoids, the operator can be written explicitly as
\begin{subequations}
\begin{align}
&\check{R}(x,d)\in \oplus_{g\in G}\End_{k}((V\otimes V)_{(d,g)})\\
&(V\otimes V)_{(d,g)}=\sum_{h\in G}V_{(d,h)}\otimes V_{(dh,h^{-1}g)},
\end{align}
\end{subequations}
for any composable edges, we have
\begin{equation}
\check{R}(x,d)|_{V_{(d,g_1)}\otimes V_{(a,g_2)}}=\oplus_{(d,g_3),(c,g_4)}\mathcal{W}\big(x;(d,g_1),(a,g_2),(d,g_3),(c,g_4)\big),
\end{equation}
where  $\mathcal{W}\big(x;(d,g_1),(a,g_2),(d,g_3),(c,g_4)\big)\in \Hom_k(V_{(d,g_1)}\otimes V_{(a,g_2)},V_{(d,g_3)}\otimes V_{(c,g_3)})$, graphically it is

\begin{equation}
\begin{tikzpicture}[scale=1.5]
	\draw[->,dashed](-.6,.5) node[left]{$\mathcal{W}\big(x;(d,g_1),(a,g_2),(d,g_3),(c,g_4)\big)=$}
	-- (1.7,.5);
	\draw[->,dashed](.5,-.6) 
	-- (.5,1.6);
	\draw[->](0,1) node[left]{$d$} -- (.5,1) node[above,fill=white]{$g_3$} -- (1,1)
	node[right]{$c$};
	\draw[->](1,1) -- (1,.5) node[right , fill=white]{$g_4$} -- (1,0);
	\draw[->](0,1) -- (0,.5) node[left, fill=white]{$g_1$} -- (0,0);
	\draw[->]node[left]{$a$}(0,0) -- (.5,0) node[below,fill=white]{$g_2$}-- (1,0)
	node[right]{$b$};
\end{tikzpicture}
\end{equation}

In the same spirit, when we look at the source fibers, we can define the various dynamical local operators.

\begin{Def}\label{def_ldHeck}
For a map $\bar{q}:\rm{Ob}(\pi)\to \C_{\ne 0}$, the local dynamical Hecke operator $S(a)$ is defined on the source fibers $s^{-1}(a)$ of $V$, $a\in \rm{Ob}(\pi)$, 
\begin{equation}
S(a)\in \oplus_{\alpha\in s^{-1}(a)}\End_{k}\big((V\otimes V)_{\alpha}\big),
\end{equation}

and satisfy the relations
\begin{subequations}
\begin{align}
	&\big(S(a)-\bar{q}(a)\big)\big(S(a)+\bar{q}^{-1}(a)\big)=0,\\
	&S^{(12)}(a)S^{(23)}(ah^{(1)})S^{(12)}(a)=S^{(23)}(ah^{(1)})S^{(12)}(a)S^{(23)}(ah^{(1)})
\end{align}
\end{subequations}
\end{Def}

\begin{Def}\label{def_ldTemp}
For a map $\bar{\kappa}:\rm{Ob}(\pi)\to \C_{\ne 0}$, an operator $T(a)$ defined on the source fibers $s^{-1}(a),a\in \rm{Ob}(\pi)$
\begin{equation}
T(a)\in \oplus_{\alpha\in s^{-1}(a)}\End_{k}\big((V\otimes V)_{\alpha}\big),
\end{equation}
is called a local dynamical Temperley--Lieb operator if it satisfies the following dynamical Temperley--Lieb equations:
\begin{subequations}\label{eq:ldTL}
\begin{align}
	&T(a)T(a)=\bar{\kappa}(a) T(a)\label{eq:ldTLa}\\
	&T^{(12)}(a)T^{(23)}(ah^{(1)})T^{(12)}(a)=T^{(12)}(a)\label{eq:ldTLb}\\
	&T^{(23)}(ah^{(1)})T^{(12)}(a)T^{(23)}(ah^{(1)})=T^{(23)}(ah^{(1)})\label{eq:ldTLc}
\end{align}
\end{subequations}
\end{Def}

\begin{Lemma}
	Suppose that we have dynamical Hecke operator $S(a)$ with parameter $q$, if $\bar{q}(a)+\bar{q}^{-1}(a)\ne 0$ for any $a$ in $\rm{Ob}(\pi)$, then we can define $T(a):=\bar{q}(a)-S(a)$, which forms the local dynamical Temperley--Lieb operator with parameter $\bar{\kappa}=\bar{q}+\bar{q}^{-1}$.
\end{Lemma}

For any finite connected unoriented graph $\Gamma$ with $n$ vertices and that has at most one edge between different vertices, let $Y\in \text{Mat}_n(\{0,1\})$ be its adjacency matrix, by the Perron--Frobenius theorem, see appendix \ref{PF_theorem}, it has a Perron--Frobenius vector $\xi$ with eigenvalue $\phi(Y)$ for $\Gamma$, and satisfies the eigenvalue equation
\begin{align*}
	Y\xi=\phi(Y)\xi
\end{align*}
For each un-oriented edge in $\Gamma$, we have two specific inverse direction edges corresponding to it in the associated groupoid $\pi(\Gamma)$, for those specific directed edges corresponding to unoriented edges, we associate a one dimensional vector space for each of them, then we have a $V\in \text{Vect}_{\pi(\Gamma)}$.

\begin{Prop}\label{ex_diagram}
Suppose that $\xi=(S_i),i\in \rm{Ob}(\pi)$, then we can construct a representation of the local dynamical Temperley--Lieb operator $T_{\Gamma}$ associated to constant function $\bar{\kappa}_{\Gamma}(a):=\phi(Y)$ by
\tikzset{>=latex}
\begin{equation}
\begin{tikzpicture}[scale=1.5]
	\draw[->,dashed](-.6,.5) node[left]{ $T_{\Gamma}(d)=\oplus_{a,c}T_{\Gamma}\Big(\begin{matrix}d&c\\a&d\\\end{matrix}\Big)=\oplus_{a,c}\sqrt{\frac{S_aS_c}{S_dS_d}}$}
	-- (1.7,.5);
	\draw[->,dashed](.5,-.6) 
	-- (.5,1.6);
	\draw[->](0,1) node[left]{$d$} -- (.5,1) node[above,fill=white]{$\gamma$} -- (1,1)
	node[right]{$c$};
	\draw[->](1,1) -- (1,.5) node[right , fill=white]{$\delta$} -- (1,0);
	\draw[->](0,1) -- (0,.5) node[left, fill=white]{$\alpha$} -- (0,0);
	\draw[->]node[left]{$a$}(0,0) -- (.5,0) node[below,fill=white]{$\beta$}-- (1,0)
	node[right]{$d$};
\end{tikzpicture}
\end{equation}
\end{Prop}

\begin{proof}
	It is easy to see that the "dynamical Temperlay--Lieb" relations are equivalent to the following graphical relations on $a$ source fibers, these graph of relations of blocks of Temperley-Lieb was already observed by the Pasquier \cite{Pasquier1987b}, with the above parametrization, these relations are easily verified.
\tikzset{>=latex}
\[ 
\begin{tikzpicture}[scale=1]
	\begin{scope}[shift={(0.5,0)}]
		\draw (-1.5,1) node{$\sum_{\beta,\gamma,d_1}$};
		\draw (2.6,1) node{$=\phi(Y)$};
		\draw[->](4,1.5) node[left]{$a$} -- (4.5,1.5) node[above,fill=white]{$\alpha_2$} -- (5,1.5)node[right]{$\bar{a}$};
		\draw[->](5,1.5) -- (5,1) node[right , fill=white]{$\beta_2$} -- (5,0.5);
		\draw[->](4,1.5) -- (4,1) node[left, fill=white]{$\alpha_1$} -- (4,0.5);
		\draw[->]node[left]{}(4,0.5) -- (4.5,0.5) node[below]{$\beta_1$}-- (5,0.5)node[right]{$a$};
		\draw (4,0.5) node[left]{\small $a_1$};
	\end{scope}
	\draw[->](0,1) node[left]{$a$} -- (.5,1) node[above,fill=white]{$\beta^{'}$} -- (1,1)node[right]{};
	\draw (1.2,0.8) node{\small $d_1$};
	\draw[->](1,1) -- (1,.5) node[right]{$\gamma$} -- (1,0);
	\draw[->](0,1) -- (0,.5) node[left, fill=white]{$\alpha_1$} -- (0,0);
	\draw[->]node[left]{$a_1$}(0,0) -- (.5,0) node[below,fill=white]{$\beta_1$}-- (1,0)node[right]{$a$};
	\draw[->](1,2) node[left]{$a$} -- (1.5,2) node[above,fill=white]{$\alpha_2$} -- (2,2)node[right]{$\bar{a}$};
	\draw[->](2,2) -- (2,1.5) node[right , fill=white]{$\beta_2$} -- (2,1);
	\draw[->](1,2) -- (1,1.5) node[left, fill=white]{$\beta$} -- (1,1);
	\draw[->]node[left]{}(1,1) -- (1.5,1) node[below]{$\gamma$}-- (2,1)node[right]{$a$};
\end{tikzpicture}
\]

\tikzset{>=latex}
\[ 
\begin{tikzpicture}[scale=1]
	\draw (2.9,1) node{$=$};
	\draw[->](4,1.5) node[left]{$a$} -- (4.5,1.5) node[above,fill=white]{$\alpha_3$} -- (5,1.5)node[right]{$\bar{a}$};
	\draw[->](5,1.5) -- (5,1) node[right , fill=white]{$\beta_4$} -- (5,0.5);
	\draw[->](4,1.5) -- (4,1) node[left, fill=white]{$\alpha_1$} -- (4,0.5);
	\draw[->]node[left]{}(4,0.5) -- (4.5,0.5) node[below]{$\beta_1$}-- (5,0.5)node[below]{$a$};
	\draw[->](5,0.5)--(5.5,0.5)node[below]{$\gamma_1$}--(6,0.5)node[right]{$a_2$};
	\draw (4,0.5) node[left]{\small $a_1$};
	\draw[->](0,1) node[left]{$a$} -- (.5,1) node[below,fill=white]{$\alpha_2$} -- (1,1)node[right]{};
	\draw (1.2,0.8) node{\small $a_2$};
	\draw[->](1,1) -- (1,.4) node[right]{$\beta_2$} -- (1,0);
	\draw[->](0,1) -- (0,.5) node[left, fill=white]{$\alpha_1$} -- (0,0);
	\draw[->]node[left]{$a_1$}(0,0) -- (.5,0) node[below,fill=white]{$\beta_1$}-- (1,0)node[below]{$a$};
	\draw[->](1,0)--(1.5,0)node[below]{$\gamma_1$}--(2.0,0)node[right]{$a_2$};
	\draw[->](2,1)--(2,0.5)node[right]{$\gamma_2$}--(2,0);
	\draw[->](1,2) node[left]{$a$} -- (1.5,2) node[above,fill=white]{$\alpha_3$} -- (2,2)node[right]{$\bar{a}$};
	\draw[->](2,2) -- (2,1.5) node[right , fill=white]{$\beta_4$} -- (2,1);
	\draw[->](1,2) -- (1,1.5) node[left, fill=white]{$\alpha_2$} -- (1,1);
	\draw[->]node[left]{}(1,1) -- (1.6,1) node[below]{$\beta_3$}-- (2,1)node[right]{$a$};
\end{tikzpicture}
\]

\end{proof}

From the duality between the square lattice model and the loop model, see \cite{Owczarek1987} section 4, we can have the following diagram interpretaion of this Temperley--Lieb operator
\begin{equation}
\begin{tikzpicture}
	\begin{scope}[shift={(4.5,0)}]
		\draw (1,2)node[above]{d} node[circle,fill,inner sep=0.5pt]{}--(0,1)node[circle,fill,inner sep=0.5pt]{} node[left]{$\oplus_{a,c}\quad$a}node[sloped,pos=0.5,allow upside down]{\arrowIn}--(1,0)node[below]{d}node[circle,fill,inner sep=0.5pt]{}node[sloped,pos=0.5,allow upside down]{\arrowIn};
		\draw (1,2)node[circle,fill,inner sep=0.5pt]{}--(2,1) node[right] {c}node[circle,fill,inner sep=0.5pt]{}node[sloped,pos=0.5,allow upside down]{\arrowIn}--(1,0)node[circle,fill,inner sep=0.5pt]{}node[sloped,pos=0.5,allow upside down]{\arrowIn};
		\draw[red,bend right=80,->](2,2) node[circle,fill,inner sep=0.5pt]{} to (2,0)node [circle,fill,inner sep=0.5pt]{};\node at (-2.2,1){$T_{\Gamma}(d)\mapsto$};
		\draw[red,bend left=80,->](0,2)node [circle,fill,inner sep=0.5pt]{} to (0,0)node [circle,fill,inner sep=0.5pt]{};	
	\end{scope}
\end{tikzpicture}
\end{equation}

And then we can form the loop model or diagram algebra definition of the operator in Proposition \ref{ex:graph} related to unoriented graphs. Later in subsection \ref{sub:relation}, we will see that the following example is a specific case of the relation between dynamical Temperley--Lieb and the usual Temperley--Lieb algebra in proposition \ref{prop:global}.

\begin{Def}
Given any connected non-oriented graph $(V,E)$ with at most single edge between two different vertices, for  $N\in \Z, N>0$ and $\phi\in \C$, the diagram algebra $\rm{dTL}(N,\phi)$ is defined as the follows.

For any $a\in V$, if there exists $b,c$ which is connected to $a$ by an edge, a generator $e_i(a)[b,c]$ is defined and the generators satisfy the following three equation

\begin{subequations}
\begin{align}
	&\sum_c e_i(a)[c,d]e_i(a)[b,c]=\phi e_i[b,d]\label{eq:TLa}\\
	&e_{i}(a)[c,d]e_{i+1}(c)[a,a]e_i(a)[b,c]=e_i(a)[b,d]\label{eq:TLb}\\
	&e_{i+1}(b)[a,d]e_i(a)[b,b]e_{i+1}(b)[c,a]=e_i(b)[c,d]\label{eq:TLc}
\end{align}
\end{subequations}

And graphically the generators $e_i(a)[b,c]$ are presented as the following:
\begin{equation}
\begin{tikzpicture}
	\draw[dashed](0,0)--(6,0)--(6,1)--(0,1)--(0,0);
	\draw(1,0)--(1,1);\draw(2,0)--(2,1);\draw(5,0)--(5,1);
	\draw[bend left=90](3,0) to (4,0);\draw[bend right=90](3,1) to (4,1);
	\draw (3,1)node[above]{$i$};\draw(4,1) node[above]{$i+1$};\draw (3,0) node[below]{$i$};\draw(4,0) node[below]{$i+1$};
	\draw (3,0.5) node{$a$};\draw(4,0.5) node{$a$};
	\draw (3.5,1.1) node[below]{\small$b$};\draw (3.5,-0.1) node [above]{\small$c$};
	\draw(-1.2,0.5) node{$e_i(a)[b,c]=$};
\end{tikzpicture}
\end{equation}
\end{Def}

The equations \eqref{eq:TLa},\eqref{eq:TLb} and \eqref{eq:TLc} are graphically described by the \eqref{eq:diagram_TLa},\eqref{eq:diagram_TLb} and \eqref{eq:diagram_TLc} respectively.

\begin{equation}\label{eq:diagram_TLa}
\begin{tikzpicture}
	\begin{scope}
		\draw[dashed](0,0)--(6,0)--(6,1)--(0,1)--(0,0);
		\draw(1,0)--(1,1);\draw(2,0)--(2,1);\draw(5,0)--(5,1);
		\draw[bend left=90](3,0) to (4,0);\draw[bend right=90](3,1) to (4,1);
		\draw (3,1)node[above]{$i$};\draw(4,1) node[above]{$i+1$};
		\draw (3,0.5) node{$a$};\draw(4,0.5) node{$a$};
		\draw (3.5,1.1) node[below]{\small$b$};\draw (3.5,-0.1) node [above]{\small$c$};
	\end{scope}
	\begin{scope}[shift={(0,-1)}]
		\draw[dashed](0,0)--(6,0)--(6,1)--(0,1)--(0,0);
		\draw(1,0)--(1,1);\draw(2,0)--(2,1);\draw(5,0)--(5,1);
		\draw[bend left=90](3,0) to (4,0);\draw[bend right=90](3,1) to (4,1);
		\draw (3,0) node[below]{$i$};\draw(4,0) node[below]{$i+1$};
		\draw (3,0.5) node{$a$};\draw(4,0.5) node{$a$};
		\draw (3.5,1.1) node[below]{\small$c$};\draw (3.5,-0.1) node [above]{\small$d$};
	\end{scope}
	\draw(6,0) node[right]{$=\phi$};
	\begin{scope}[shift={(7.0,-0.5)}]
		\draw[dashed](0,0)--(6,0)--(6,1)--(0,1)--(0,0);
		\draw(1,0)--(1,1);\draw(2,0)--(2,1);\draw(5,0)--(5,1);
		\draw[bend left=90](3,0) to (4,0);\draw[bend right=90](3,1) to (4,1);
		\draw (3,1)node[above]{$i$};\draw(4,1) node[above]{$i+1$};\draw (3,0) node[below]{$i$};\draw(4,0) node[below]{$i+1$};
		\draw (3,0.5) node{$a$};\draw(4,0.5) node{$a$};
		\draw (3.5,1.1) node[below]{\small$b$};\draw (3.5,-0.1) node [above]{\small$d$};
	\end{scope}
\end{tikzpicture}
\end{equation}

\begin{equation}\label{eq:diagram_TLb}
\begin{tikzpicture}
	\begin{scope}
		\draw[dashed](0,0)--(6,0)--(6,1)--(0,1)--(0,0);
		\draw(1,0)--(1,1);\draw(2,0)--(2,1);\draw(5,0)--(5,1);
		\draw[bend left=90](3,0) to (4,0);\draw[bend right=90](3,1) to (4,1);
		\draw (3,1)node[above]{$i$};\draw(4,1) node[above]{$i+1$};
		\draw (3,0.5) node{$a$};\draw(4,0.5) node{$a$};
		\draw (3.5,1.1) node[below]{\small$b$};\draw (3.5,-0.1) node [above]{\small$c$};
	\end{scope}
	
	\begin{scope}[shift={(0,-1)}]
		\draw[dashed](0,0)--(6,0)--(6,1)--(0,1)--(0,0);
		\draw(1,0)--(1,1);\draw(2,0)--(2,1);\draw(3,0)--(3,1);
		\draw[bend left=90](4,0) to (5,0);\draw[bend right=90](4,1) to (5,1);
		\draw (4,0.5) node{$c$};\draw(5,0.5) node{$c$};
		\draw (4.5,1.1) node[below]{\small$a$};\draw (4.5,-0.1) node [above]{\small$a$};
	\end{scope}
	
	\begin{scope}[shift={(0,-2)}]
		\draw[dashed](0,0)--(6,0)--(6,1)--(0,1)--(0,0);
		\draw(1,0)--(1,1);\draw(2,0)--(2,1);\draw(5,0)--(5,1);
		\draw[bend left=90](3,0) to (4,0);\draw[bend right=90](3,1) to (4,1);
		\draw (3,0) node[below]{$i$};\draw(4,0) node[below]{$i+1$};
		\draw (3,0.5) node{$a$};\draw(4,0.5) node{$a$};
		\draw (3.5,1.1) node[below]{\small$c$};\draw (3.5,-0.1) node [above]{\small$d$};
	\end{scope}
	
	\draw(6,-0.5) node[right]{$=$};
	\begin{scope}[shift={(6.5,-1)}]
		\draw[dashed](0,0)--(6,0)--(6,1)--(0,1)--(0,0);
		\draw(1,0)--(1,1);\draw(2,0)--(2,1);\draw(5,0)--(5,1);
		\draw[bend left=90](3,0) to (4,0);\draw[bend right=90](3,1) to (4,1);
		\draw (3,1)node[above]{$i$};\draw(4,1) node[above]{$i+1$};\draw (3,0) node[below]{$i$};\draw(4,0) node[below]{$i+1$};
		\draw (3,0.5) node{$a$};\draw(4,0.5) node{$a$};
		\draw (3.5,1.1) node[below]{\small$b$};\draw (3.5,-0.1) node [above]{\small$d$};
		\draw (5.5,0.5) node{$c$};
	\end{scope}
\end{tikzpicture}
\end{equation}

\begin{equation}\label{eq:diagram_TLc}
\begin{tikzpicture}
	\begin{scope}
		\draw[dashed](0,0)--(6,0)--(6,1)--(0,1)--(0,0);
		\draw(1,0)--(1,1);\draw(2,0)--(2,1);\draw(5,0)--(5,1);
		\draw[bend left=90](3,0) to (4,0);\draw[bend right=90](3,1) to (4,1);
		\draw (3,1)node[above]{$i+1$};\draw(4,1) node[above]{$i+2$};
		\draw (3,0.5) node{$b$};\draw(4,0.5) node{$b$};
		\draw (3.5,1.1) node[below]{\small$c$};\draw (3.5,-0.1) node [above]{\small$a$};
	\end{scope}
	
	\begin{scope}[shift={(0,-1)}]
		\draw[dashed](0,0)--(6,0)--(6,1)--(0,1)--(0,0);
		\draw(1,0)--(1,1);\draw(4,0)--(4,1);\draw(5,0)--(5,1);
		\draw[bend left=90](2,0) to (3,0);\draw[bend right=90](2,1) to (3,1);
		\draw (2,0.5) node{$a$};\draw(3,0.5) node{$a$};
		\draw (2.5,1.1) node[below]{\small$b$};\draw (2.5,-0.1) node [above]{\small$c$};
	\end{scope}
	
	\begin{scope}[shift={(0,-2)}]
		\draw[dashed](0,0)--(6,0)--(6,1)--(0,1)--(0,0);
		\draw(1,0)--(1,1);\draw(2,0)--(2,1);\draw(5,0)--(5,1);
		\draw[bend left=90](3,0) to (4,0);\draw[bend right=90](3,1) to (4,1);
		\draw (3,0) node[below]{$i+1$};\draw(4,0) node[below]{$i+2$};
		\draw (3,0.5) node{$b$};\draw(4,0.5) node{$b$};
		\draw (3.5,1.1) node[below]{\small$a$};\draw (3.5,-0.1) node [above]{\small$d$};
	\end{scope}
	
	\draw(6.0,-0.5) node[right]{$=$};
	\begin{scope}[shift={(6.5,-1)}]
		\draw[dashed](0,0)--(6,0)--(6,1)--(0,1)--(0,0);
		\draw(1,0)--(1,1);\draw(2,0)--(2,1);\draw(5,0)--(5,1);
		\draw[bend left=90](3,0) to (4,0);\draw[bend right=90](3,1) to (4,1);
		\draw (3,1)node[above]{$i+1$};\draw(4,1) node[above]{$i+2$};\draw (3,0) node[below]{$i+1$};\draw(4,0) node[below]{$i+2$};
		\draw (3,0.5) node{$b$};\draw(4,0.5) node{$b$};
		\draw (3.5,1.1) node[below]{\small$c$};\draw (3.5,-0.1) node [above]{\small$d$};
		\draw (1.5,0.5) node{$a$};
	\end{scope}
\end{tikzpicture}
\end{equation}

\begin{Def}\label{def_ldBMW}
For two maps $\bar{q},\bar{\nu}:\rm{Ob}(\pi)\to \C^{\times}$ and an invertible operator $U(a)$ defined on the source fiber $s^{-1}(a)$ of $V$ where $a\in \rm{Ob}(\pi)$. That is
\begin{equation}
U(a)\in \oplus_{\alpha\in s^{-1}(a)}\End_{k}\big((V\otimes V)_{\alpha}\big)
\end{equation}
and we denote
\begin{equation}
K(a)=\id-\frac{U(a)-U^{-1}(a)}{\bar{q}(a)-\bar{q}^{-1}(a)},
\end{equation}
$U(a)$ is called the local dynamical Birman-Murakami-Wenzl operator associated to $\bar{q},\bar{\nu}$, if they satisfies the following relations
\begin{subequations}
\begin{align}
	&U^{(12)}(a)U^{(23)}(ah^{(1)})U^{(12)}(a)=U^{(23)}(ah^{(1)})U^{(12)}(a)U^{(23)}(ah^{(1)})\\
	&K(a)U(a)=U(a)K(a)=\bar{\nu}(a) K(a)\\
	&K^{(23)}(ah^{(1)})(U^{\epsilon})^{(12)}(a)K^{(23)}(ah^{(1)})=\bar{\nu}^{-\epsilon}(ah^{(1)})K^{(23)}(ah^{(1)}),\quad \epsilon=\pm 1\\
	&K^{(12)}(a)(U^{\epsilon})^{(23)}(ah^{(1)})K^{(12)}(a)=\bar{\nu}^{-\epsilon}(a)K(a),\quad \epsilon=\pm 1
\end{align}
\end{subequations}
\end{Def}

\subsection{Relation between the usual and dynamical operators}\label{sub:relation}
There are three types of relations that relate the non-dynamical and dynamical operators.

The \emph{first type} of relation is about the groupoid structure, suppose that the groupoid $\pi$ is trivial that has only one vertex and one identity arrow, then the category of $\pi$ graded vector space becomes just the usual category of $k$ vector space
\[
\rm{Vect}_k(\pi)\simeq \rm{Vect}_k,
\]
the local dynamical operators does not depend on the dynamical shift and becomes the usual operators, for example the dynamical Yang-Baxter equation \eqref{eq:dYBE}
becomes the usual quantum Yang-Baxter equation
\[
\check{R}(z-w)^{(23)}\check{R}(z)^{(12)}\check{R}(w)^{(23)}=\check{R}(w)^{(12)}\check{R}(z)^{(23)}\check{R}(z-w)^{(12)},
\] 

and similarly other local dynamical operators will become the local non dynamical operators.

The \emph{second type} relation is about the relations between usual operators  on groupoid graded representation and dynamical operators on the source fibers. We can simply restrict the operators to the source fibers to get the dynamical operators.

For example, if $\pi=A\rtimes G$ is an action groupoid with an Yang-Baxter operator $\check{R}(x)$ defined on $V\in \text{Vect}_k(\pi)$, we can restrict the $\check{R}(x)$ to a graded component of the vector space with fixed $a\in A$
\begin{align*}
	\check{R}(x,a):=\check{R}(x)|_{ \oplus_{g\in G}\End_{k}((V\otimes V)_{(a,g)})},
\end{align*} 
the restriction of the Yang-Baxter equation to the graded component will be the corresponding dynamical Yang-Baxter equation.

The third type of relation is the globalization which goes from dynamical to non-dynamical, for example we can first define the following global dynamical operators.

\begin{Def}
	Let $V_i\in \text{Vect}_k(\pi)$, $\bar{q}_i:\rm{Ob}(\pi)\to \C_{\ne 0}$, $i=1,\dots,N$, for operators $S_i(a)$ defined on the source fibers $s^{-1}(a), a\in \text{Ob}(\pi)$,
	\begin{equation}
	S_i(a)\in \oplus_{\alpha\in s^{-1}(a) }\End_{k}\big((\otimes^N_{t=1}V_i)_{\alpha}\big),
	\end{equation}	
	they forms a dynamical Hecke operator associated to $\bar{q}_i$, if it satisfies the following dynamical Hecke relations:
	\begin{subequations}
	\begin{align}
		&[S_i(a),\bar{q}_i(ah^{i-1})]=0\\
		&\big(S_i(a)-\bar{q}_i(ah^{i-1})\big)\big(S_i(a)+\bar{q}_i^{-1}(ah^{i-1})\big)=0,\quad  1\le i\le N-1\\
		&S_i(a)S_{i+1}(a)S_i(a)=S_{i+1}(a)S_i(a)S_{i+1}(a),\quad 1\le i\le N-2\\
		&S_i(a)S_j(a)=S_j(a)S_i(a),\quad |i-j|>1
	\end{align}
	\end{subequations}	
	Many classical arguments naturally extend to the dynamical case, for example, we can also define the Murphy element of type $A$, the type here is actually refer to both the boundary type and the Lie algebra type. Here is the definition,
	\begin{align}
		&J_1^{(A)}(a):=S^{2}_1(a);\\
		&J_i^{(A)}(a):=S_i(a)J^{(A)}_{i-1}(a)S_i(a),\quad 2\le i\le N-1
	\end{align}
	
\end{Def}

\begin{Prop}
	Suppose that $\bar{q}_1=\bar{q}_2=\dots=\bar{q}_N$ is constant, the Murphy elements satisfies the relations:
	\begin{align}
		&[J^{(A)}_i(a),J^{(A)}_j(a)]=0\\
		&[S_1(a),J^{(A)}_j(a)]=0, \quad j\ge 1\\
		&[S_i(a),J^{(A)}_j(a)]=0,\quad j>i, i\ge 2,j\ne i-1,i\\
		&[S_i(a),J^{(A)}_{i-1}(a)J^{(A)}_i(a)]=0,\quad i\ge 2\\
		&[S_i(a), J^{(A)}_i(a)+J^{(A)}_{i-1}(a)]\quad i\ge 2\\
	\end{align}
\end{Prop}

\begin{Rem}
	For different boundary conditions, there may be other different generalizations of Temperley--Lieb and Hecke algebra as in the classical case, for example as in \cite{DeGier2009}.
\end{Rem}

\begin{Def}\label{def:dynamical_glocal_TL}
	Let $V_i \in \text{Vect}_k(\pi)$, $\bar{\kappa}_i:\rm{Ob}(\pi)\to \C_{\ne 0}$, $i=1,\dots,N$, the operators $T_i(a)$ are defined on the source fibers $s^{-1}(a), a\in \text{Ob}(\pi)$,
	\begin{equation}
	T_i(a)\in \oplus_{\alpha\in s^{-1}(a) }\End_{k}\big((\otimes^N_{t=1}V_i)_{\alpha}\big),
	\end{equation}	
	they forms a dynamical Temperley--Lieb algebra if it satisfies the following dynamical Hecke relations:
	\begin{subequations}
    \begin{align}
    	&T_i(a)T_i(a)=\bar{\kappa}_i(ah^{i-1})T_i(a),\quad i=1,\dots,N-1\\
    	&T_i(a)T_{i+1}(a)T_i(a)=T_i(a),\quad 1\le i\le N-2\\
    	&T_{i+1}(a)T_i(a)T_{i+1}(a)=T_{i+1}(a),\quad 1\le i\le N-2\\
    	&T_i(a)T_j(a)=T_j(a)T_i(a),\quad |i-j|> 1
    \end{align}
	\end{subequations}	
\end{Def}

\begin{Lemma}
	Suppose that we have dynamical Hecke operators $S_i(a)$ associated to $\bar{q}_i$, if $\bar{q}_i(ah^{i-1})+\bar{q}_i^{-1}(ah^{i-1})$ is not equal to zero for any $a\in \rm{Ob}(\pi)$ and $i=1,\dots,N$, then we can define $T_i(a):=\bar{q}_i(ah^{(i)})-S_i(a)$, which forms the dynamical Temerpey-Lieb operator associated to $\bar{\kappa}_i=\bar{q}_i+\bar{q}_i^{-1}$.
\end{Lemma}

The following proposition describes the relation between the local dynamical operators and global dynamical operators.

\begin{Prop}\label{prop:global}
Let $T(a)$ be a local dynamical operator defined on source fibers of $V$ in definition \ref{def_ldTemp} associated to $\bar{\kappa}$, then
\begin{equation}
	T_{i}(a):=T^{(i,i+1)}(ah^{i-1})
\end{equation}
is a representation of dynamical Temperley-Lieb operator on the source fiber space of $V^{\otimes N}$ associated with $\bar{\kappa}_i=\bar{\kappa},i=1,\dots,N$.
\end{Prop}

The following proposition describes the relation between the "global" dynamical algebra and the usual algebra.
\begin{Prop}
Suppose that $\bar{\kappa}_1=\bar{\kappa}_2=\dots=\bar{\kappa}_N$ is a constant function for all $a$ in the Definition \ref{def:dynamical_glocal_TL} , then collecting all the fiber space, we get the operator $T_i:=\oplus T_i(a)$, which forms a representation on groupoid graded vector space of usual Temperley-Lieb algebras associated with the constant $\bar{\kappa}(ah^{(i-1)})$.
\end{Prop}

And similarly for the Birman-Wenzl-Mirakami case, we give the following definitions of the global version.

\begin{Def}\label{dBMW}
	Let $V_i\in \text{Vect}_k(\pi),\bar{q}_i,\bar{\nu}_i:\rm{Ob}\to \C_{\ne 0}$,$i=1,\dots,N$, the invertible operators $U_i(a),i=1,\dots,N-1$ are defined on source fibers $s^{-1}(a), a\in \text{Ob}(\pi)$,
	\begin{equation}
	U_i(a)\in \oplus_{\alpha\in s^{-1}(a) }\End_{k}\big((\otimes^N_{t=1}V_i)_{\alpha}\big).
	\end{equation}    
    The $K_i(a)$ are defined by 
		\begin{align}
		K_i(a)=\id-\big(q(ah^{i-1})-q^{-1}(ah^{i-1}))^{-1}\big(U_i(a)-U^{-1}_i(a)\big).
	\end{align}
	
	 $U_i(a)$ are called dynamical Birman-Murakami-Wenzl operator if they satisfies the relation
	 \begin{subequations}
	 \begin{align}
	 	&U_i(a)U_{i+1}(a)U_i(a)=U_{i+1}(a)U_i(a)U_{i+1}(a)\\
	 	&K_i(a)U_i(a)=U_i(a)K_i(a)=\nu_i(ah^{i-1}) K_i(a)\\
	 	&K_i(a)U^{\epsilon}_{i-1}(a)K_i(a)=\nu^{-\epsilon}(ah^{i-1})K_i(a),\quad \epsilon=\pm 1\\
	 	&K_i(a)T^{\epsilon}_{i+1}(a)K_i(a)=\nu^{-\epsilon}(ah^{i-1})K_i(a),\quad \epsilon=\pm 1\\
	 \end{align}
	 \end{subequations}
	
\end{Def}

\subsection{Baxterization}\label{sub:Baxterization}
In this section, we consider the Baxterization process of the local dynamical operators.

In the first case, let $V\in \text{Vect}_k(\pi)$, suppose that we have invertible operators $\sigma(a)\in \oplus_{\alpha\in s^{-1}(a)}\End_{k}\big((V\otimes V)_{\alpha}\big)$ that satisfy the relations

\begin{subequations}
\begin{align}
&\sigma^{(12)}(a)\sigma^{(23)}(ah^1)\sigma^{(12)}(a)=\sigma^{(23)}(ah^1)\sigma^{(12)}(a)\sigma^{(23)}(ah^1)\\
&\sigma(a)+\sigma^{-1}(a)=f(a)\id
\end{align}
\end{subequations}
where $f:\text{Ob}(\pi)\to \C_{\ne 0}$.

\begin{Th}\label{th:baxter_Hecke}
Suppose that $f(a)=f(b)$ if there exists an arrow $\alpha$ with $s(\alpha)=a,t(\alpha)=b$, then the operator defined by
\begin{align*}
\check{R}(z,a)=e^z\sigma(a)+e^{-z}\sigma^{-1}(a)
\end{align*}
satisfies the dynamical Yang-Baxter equation \eqref{eq:dYBE}.
\end{Th}

\begin{proof}
We use the relations
\[
\sigma(a)+\sigma^{-1}(a)=f(a)\id,\quad \sigma(ah^1)+\sigma^{-1}(ah^1)=f(ah^1)\id=f(a)\id
\]
\end{proof}

In the second case, if we assume that 
\begin{align}\label{eq:ansatz_dTL}
	\check{R}(x,a)=\id+xT(a)
\end{align}
where $T(a)$ is local dynamical Temperley-Lieb operators associated to $\bar{\kappa}$ on $V$ as defined in \ref{def_ldTemp} then we will get

\begin{Th}\label{th:baxter_lTl}
Suppose that $x=f(z),x'=f(z'),x''=f(z''),z''=z'-z$ satisfies the following equation
\begin{align}\label{eq:relation}
	x^{''}=\frac{x'-x}{1+\bar{\kappa}(ah^1)x+xx'},\quad x^{''}=\frac{x'-x}{1+\bar{\kappa}(a)x+xx'}
\end{align}
then the operators $\check{R}(x,a)=\id+xT(a)$ satisfies the dynamical Yang-Baxter equation \eqref{eq:gdYBE}.
\end{Th} 
\begin{proof}
Inserting the assumption \eqref{eq:ansatz_dTL} into the dynamical Yang--Baxter equation and using the relations of local Temperlay--Lieb operator, we get the following relations
\begin{equation}
\begin{split}
&(x''+x+\bar{\kappa}(ah^1) xx''+xx'x''-x')T^{(23)}(ah^{(1)})\\
&=(x''+x+\bar{\kappa}(a)xx''+xx'x''-x')T^{(12)}(a),
\end{split}
\end{equation}
so suppose that we have the relation \eqref{eq:relation}, then the Yang-Baxter relation will be satisfied.
\end{proof}

In the third case, we consider about the Birman-Murakami-Wenzl case.
\begin{Th}\label{th:Baxter_BMW}
Suppose that the operators $U(a)$ are the local dynamical Birman-Murakami-Wenzl operator associated with $\bar{q},\bar{\nu}$ on $V$, and suppose that $
\bar{q}(a)=\bar{q}(b),\bar{\nu}(a)=\bar{\nu}(b),$ if there exists an arrow $\alpha\in \pi$ with $s(\alpha)=a,t(\alpha)=b$. Then we define
\begin{align}
\check{R}(u,v)[a]:=U(a)+\frac{\bar{q}(a)-\bar{q}^{-1}(a)}{v/u-1}+\frac{\bar{q}(a)-\bar{q}^{-1}(a)}{1+\bar{\nu}^{-1}(a)\bar{q}(a)v/u}K(a),
\end{align}
and it satisfies the following two parameters dynamical Yang-Baxter equation
\begin{equation}
\begin{split}
&\check{R}^{(12)}(u_2,u_3)[a]\check{R}^{(23)}(u_1,u_3)[ah^{(1)}]\check{R}^{(12)}(u_1,u_2)[a]\\
&=\check{R}^{(23)}(u_1,u_2)[ah^{(1)}]\check{R}^{(12)}(u_1,u_3)[a]\check{R}^{(23)}(u_2,u_3)[ah^{(1)}]
\end{split}
\end{equation}
\end{Th}

\begin{proof}
With the assumption $
\bar{q}(a)=\bar{q}(b),\bar{\nu}(a)=\bar{\nu}(b),$ if there exists an arrow $\alpha\in \pi$ with $s(\alpha)=a,t(\alpha)=b$, the proof is similar to the classical case.
\end{proof}

\subsection{Example: unrestricted cases}\label{sub:example_ellptic}
The first case we consider is the unrestricted case. For the non-oriented graph  with objects $(\Z+b)$ and an edge for every two neighbouring numbers (see Figure \ref{figure:infinite groupoid}),here $b$ is a generic shift to avoid the singularities of the operators we defined below,  
\begin{figure}[h]
	\centering
	\begin{tikzpicture}
		\draw (-3,0) node{\large $\pi^{\text{unres}}_A:$};
		\draw (-1,0)--(0,0)--(1,0)--(2,0)--(3,0)--(4,0)--(5,0)--(6,0);
		\node at (-1,0)[circle,fill,inner sep=1.5pt]{};
		\draw (-1,-0.3) node{\small $\dots$};
		\node at (0,0)[circle,fill,inner sep=1.5pt]{};
		\draw (0,-0.3) node{\small $1+b$};
		\node at (1,0)[circle,fill,inner sep=1.5pt]{};
		\draw (1,-0.3) node{\small $2+b$};
		\node at (2,0)[circle,fill,inner sep=1.5pt]{};
		\draw (2,-0.3) node{\small $\dots$};
		\node at (3,0)[circle,fill,inner sep=1.5pt]{};
		\node at (4,0)[circle,fill,inner sep=1.5pt]{};
		\node at (5,0)[circle,fill,inner sep=1.5pt]{};
		\draw (5,-0.3) node{\small $n+b$};
		\node at (6,0)[circle,fill,inner sep=1.5pt]{};
		\draw (6,-0.3) node{\small $\dots$};
	\end{tikzpicture}
	\caption{unrestricted groupoid of type $A$}\label{figure:infinite groupoid}
\end{figure}
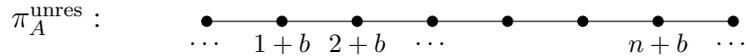

the corresponding groupoid is denoted by  $\pi_A^{\text{unres}}$. And we define the following $\pi_A^{\text{unres}}$ graded vector space,
\begin{align}
	V^{\pi^{\text{unres}}_A}_{(a,-1)}=\C e_{(a,-1)},\quad V^{\pi^{\text{unres}}_A}_{(a,+1)}=\C e_{(a,+1)}. 
\end{align} 
here $(a,-1)$ is the edge with source a and target $(a-1)$ and $(a,+1)$ is the edge with source $a$ and target $(a+1)$.

In this case, we can identify the fiber space with some vector spaces and then write the operator in a matrix form. For any $a\in \pi^{\rm{unres}}$,
\begin{equation}\label{eq:fiberspace}
	\oplus_{\alpha\in s^{-1}(a)}(V^{\pi^{\text{unres}}_A}\otimes V^{\pi^{\text{unres}}_A})_{\alpha}\cong \C^4
\end{equation}
by identifying $e_{(a,+1)}\otimes e_{(a+1,+1)}$ with $(1,0)^T\otimes (1,0)^T$, $e_{(a,+1)}\otimes e_{(a+1,-1)}$ with $(1,0)^T\otimes (0,1)^T$,$e_{(a,-1)}\otimes e_{(a-1,+1)}$ with $(0,1)^T\otimes (1,0)^T$ and $e_{(a,-1)}\otimes e_{(a-1,-1)}$ with $(0,1)^T\otimes (0,1)^T$, let $E_{ij}$ be the $2\times 2$ matrix unit such that $E_{ij}e_k=\delta_{jk}e_i$ for all $k\in \{1,2\}$.

In the following, we define operators $T^{\text{tri}}_A(a), T^{\text{hyb}}_A(a), T^{\text{ell}}_A(a)$ on $\oplus_{\alpha\in s^{-1}(a)}\End_{k}\big((V^{\pi^{\text{unres}}_A}\otimes V^{\pi^{\text{unres}}_A})_{\alpha}\big)$.

\begin{Prop}\label{prop:tri_unres}
	Let $\langle z\rangle:=\sin(\frac{\pi z}{L+1})$, the operators 
	\begin{equation}
		\begin{split}
			T^{\text{tri}}_A(a)&=\frac{\sqrt{\langle a-1\rangle\langle a+1\rangle}}{\langle a\rangle}E_{21}\otimes E_{12}+\frac{\sqrt{\langle a+1\rangle\langle a-1\rangle}}{\langle a\rangle}E_{12}\otimes E_{21}
			\\
			+&\frac{\langle a+1\rangle}{\langle a\rangle}E_{11}\otimes E_{22}+\frac{\langle a-1\rangle}{\langle a\rangle}E_{22}\otimes E_{11}
		\end{split}
	\end{equation}
	forms local dynamical Temperley--Lieb operator algebra on $V^{\pi^{\text{unres}}_A}$ associated with the constant function $\bar{\kappa}^{\text{tri}}=2\cos \lambda,\lambda=\pi/(L+1)$.
\end{Prop}

\begin{Lemma}
The function $x=\frac{\langle z\rangle}{\langle 1-z \rangle}$ satisfies the equation \eqref{eq:relation} for the $\bar{\kappa}^{\text{tri}}$, then it follows $\check{R}^{\text{tri}}_A(z,a)=\id+\frac{\langle z \rangle }{\langle 1-z \rangle} T^{\text{tri}}_A(a)$ satisfies the dynamical YBE.
\end{Lemma}

\begin{Prop}\label{prop:hyp_unres}
Let $\ldb z\rdb:=\sinh(\frac{\pi z}{L+1})$, the operators 
\begin{equation}
	\begin{split}
		T_A^{\text{hyp}}(a)&=\frac{\sqrt{\ldb a-1\rdb \ldb a+1\rdb}}{\ldb a\rdb}E_{21}\otimes E_{12}+\frac{\sqrt{\ldb a+1\rdb \ldb a-1\rdb}}{\ldb a\rdb}E_{12}\otimes E_{21}\\
		&+\frac{\ldb a+1\rdb}{\ldb a\rdb}E_{11}\otimes E_{22}+\frac{\ldb a-1\rdb}{\ldb a\rdb}E_{22}\otimes E_{11}
	\end{split}
\end{equation} 
forms local dynamical Temperley--Lieb operator algebra on $V^{\pi^{\text{unres}}_A}$ associated with the constant function $\bar{\kappa}^{\text{hyp}}(a)=2\cosh(\frac{\pi}{L+1})$.
\end{Prop}

\begin{Lemma}
	The function $x=\frac{\ldb z\rdb}{\ldb 1-z \rdb}$ satisfies the equation \eqref{eq:relation} for the $\bar{\kappa}^{\text{hyb}}$, then it follows $\check{R}^{\text{hyb}}_A(z,a)=\id+\frac{\ldb z \rdb }{\ldb 1-z \rdb} T^{\text{hyb}}_A(a)$ satisfies the dynamical YBE.
\end{Lemma}

Fix complex numbers $\tau$ and $L\in \Z_{\ge 2}$ such that $\imm \tau>0$ and $\frac{1}{L+1}\notin \Z+\tau\Z$, let
\begin{align}
	\theta(z,\tau)=-\sum_{n\in \Z}e^{i\pi(n+\frac{1}{2})^2\tau+2\pi i(n+\frac{1}{2})(z+\frac{1}{2})}
\end{align}
be the odd Jacobi theta function and $[z]=\theta\big( z/(L+1),\tau\big)/\big(\theta^{'}(0,\tau)/(L+1)\big)$ is normalized to have derivative 1 at $z=0$. 

\begin{Prop}\label{prop:ell_unres}
	On the source fiber vector space of $V^{\pi^{\rm{unres}}_A}$, the operators $T^{\rm{ell}}_A$ defined by 
	\begin{equation}
		\begin{split}
			T^{\rm{ell}}_A(a)&=\frac{\sqrt{[a-1][a+1]}}{[a]}E_{21}\otimes E_{12}+\frac{\sqrt{[a+1][a-1]}}{[a]}E_{12}\otimes E_{21}
			\\
			+&\frac{[a+1]}{[a]}E_{11}\otimes E_{22}+\frac{[ a-1]}{[ a]}E_{22}\otimes E_{11}
		\end{split}
	\end{equation}
	forms local dynamical Temperley--Lieb operator associated with map
	\begin{equation}
		\bar{\kappa}^{\text{ell}}(a)=\frac{[a+1]+[a-1]}{[a]}
	\end{equation}
\end{Prop}

\begin{proof}
	The first equation \eqref{eq:ldTLa} is directly computed and it will produce the $\bar{\kappa}^{\text{ell}}$ map. For the equation \eqref{eq:ldTLb} and \eqref{eq:TLc}, same as proof \ref{ex_diagram}, it is directly checked.
\end{proof}

In the elliptic case, the dynamical $\check{R}$ matrix related to $T^{\text{ell}}_A$ is Andrews-Baxter-Forrester \cite{Andrews1984} parametrization of  elliptic dynamical \cite{Felder1994} $R$ matrix 
\begin{equation}
	\begin{split}
		\check{R}^{\text{ell}}_A(z,a)&=\sum_{i=1}^{2}E_{ii}\otimes E_{ii}+\frac{\sqrt{[a-1][a+1]}[z]}{[a][1-z]}E_{21}\otimes E_{12}+\frac{\sqrt{[a+1][a-1]}[z]}{[a][1-z]}E_{12}\otimes E_{21}
		\\
		+&\frac{[a+z][1]}{[a][1-z]}E_{11}\otimes E_{22}+\frac{[a-z][1]}{[a][1-z]}E_{22}\otimes E_{11}
	\end{split}
\end{equation}

\begin{Lemma}
	Taking the trigonometric limit  $\tau\to -i\infty$  of $\check{R}^{\text{ell}}_A$ and $\check{T}^{\text{ell}}$, we get the corresponding $\check{R}^{\text{hyb}}_A$ and $\check{T}^{\text{hyb}}_A$.
\end{Lemma}

\begin{proof}
we have the following addition formula
\begin{align*}
	&\langle a+z\rangle\langle 1\rangle=\sin(\frac{\pi(a+z)\pi}{L+1})\sin(\frac{\pi}{L+1})\\
	&=(\sin(\frac{a\pi}{L+1})\cos(\frac{z\pi}{L+1}))\sin(\frac{\pi}{L+1})+\sin(\frac{z\pi}{L+1})\cos(\frac{a\pi}{L+1})\sin(\frac{\pi}{L+1})\\
	&-\sin(\frac{z\pi}{L+1})\cos(\frac{\pi}{L+1})\sin(\frac{a\pi}{L+1})+\sin(\frac{z\pi}{L+1})\cos(\frac{\pi}{L+1})\sin(\frac{a\pi}{L+1})\\
	&=\sin(\frac{a\pi}{L+1})(\sin(\frac{\pi}{L+1}-\frac{z\pi}{L+1}))+\sin(\frac{\pi z}{L+1})\sin\frac{(\pi(a+1))}{L+1}
\end{align*}
and similarly for $\langle a-z\rangle\langle 1\rangle$.
\end{proof}
\begin{Rem}
There are two new phenomenons in the elliptic case, the first is that the function $\kappa$ is not constant, the second is that $\check{T}^{\text{ell}}_A$ is not Baxterized in the sense of \ref{th:baxter_lTl}.  This can be seen as the mathematical interpretation of the physics fact that only passing to the trigonometric or hyperbolic limit of the lattice model, we have some critical behaviors.
\end{Rem}

\subsection{Example: restricted case}\label{sub:ADE}
In this section, we recover reformulate the classical results of Pasquier models \cite{Pasquier1987} and Temperley--Lieb interaction models \cite{Owczarek1987}, see also \cite{PEARCE1990}.

We first discuss the classical $ADE$ case, let $\Gamma$ be one of the classical $ADE$ diagrams. From the Proposition \ref{ex_diagram}, we get the dynamical Temperley--Lieb operator corresponding to those graphs. The Perro-Frobenius eigenvalue of the graphs are $2\cos(\lambda)$ where $\lambda=\pi/h$ and $h$ is the coexter number for the ADE graph as in the Table \ref{coexter_number}:
\begin{table}[h]
	\begin{center}
		\begin{tabular}{|c|c|c|c|c|c|}
			\hline
			Lie algebra& $A_L$   & $D_L$ &$E_6$&$E_7$&$E_8$\\
			\hline
			Coxter number& $L+1$ & $2L-2$&$12$ &$18$&$30$\\
			\hline
		\end{tabular}
	\end{center}
	\caption{table of coexter number}\label{coexter_number}
\end{table}

\begin{Prop}\label{prop:res_ade}
The parameterization $x=\frac{\sin z}{\sin(\lambda-z)}$ satisfies the relation \eqref{eq:relation} for $\bar{\kappa}_{\Gamma}$, the dynamical $\check{R}(z,a):=\id+\frac{\sin z}{\sin (\lambda-z)}T_{\Gamma}(a)$ satisfies the dynamical YBE.
\end{Prop}

Let $\Gamma^{(1)}$ denote one of the affine $ADE$ graphs, from the Proposition \ref{ex_diagram}, we get the dynamical Temperley--Lieb operator corresponding to those graphs, the Perro-Frobenius eigenvalue of the graphs are 2.

\begin{Prop}\label{prop:res_affineade}
The parameterization $x=\frac{z}{1-z}$ satisfies the relation \eqref{eq:relation} for $\kappa_{\Gamma^{(1)}}$, the dynamical $\check{R}(z,a):=\id+\frac{z}{1-z}T_{\Gamma^{(1)}}$ satisfies the dynamical YBE.
\end{Prop}

For the type $A$ case, we can write the $\check{R}$ matrix and operator $T$ more explicitly as a matrix similar to the unrestricted setting, but because our groupoid graded vector space is the restricted case, we do not have the isomorphism for all fiber space \ref{eq:fiberspace}, so there are some items may not appear due to the restriction.
	
We denote $\langle a\rangle:=\sin(a\pi/(L+1))$, for $3\le a\le L-3$, we have the following
	\begin{align}
		\check{R}_{\Gamma_A}(u,a)=
		\begin{pmatrix}
			1 &0   &0   &0\\
			0 &1+\frac{\sin (u)\langle a+1\rangle}{\sin(\lambda-u)\langle a\rangle}& \frac{\sin u\sqrt{\langle a-1\rangle\langle a+1\rangle}}{\sin (\lambda-u)\langle a\rangle} &0\\
			0& \frac{\sin (u)\sqrt{\langle a-1\rangle\langle a+1\rangle}}{\sin (\lambda-u)\langle a\rangle} &1+\frac{\sin (u)\langle a-1\rangle}{\sin(\lambda-u)\langle a\rangle}&0\\
			0&0&0&1
		\end{pmatrix}
	\end{align}
	
	And in this case the $T_{\Gamma_A}$ is the following:
	\begin{align*}
		T_{\Gamma_A}(a):=
		\begin{pmatrix}
			0 &0   &0   &0\\
			0 &\frac{\langle a+1\rangle}{\langle a\rangle}& \frac{\sqrt{\langle a-1\rangle\langle a+1\rangle}}{\langle a\rangle} &0\\
			0& \frac{\sqrt{\langle a-1\rangle\langle a+1\rangle}}{\langle a\rangle} &\frac{\langle a-1\rangle}{\langle a\rangle}&0\\
			0&0&0&0
		\end{pmatrix}
	\end{align*}
	
	For other $a$ near the boundary, certain terms in these matrix are not defined because the fiber space do not have the isomorphism \eqref{eq:relation}. For example, if $a=2$,
	\begin{align*}
		\check{R}_{\Gamma_A}(u,a)=
		\begin{pmatrix}
			1 &0   &0   &0\\
			0 &1+\frac{\sin (u)\langle a+1\rangle}{\sin(\lambda-u)\langle a\rangle}& \frac{\sin u\sqrt{\langle a-1\rangle\langle a+1\rangle}}{\sin (\lambda-u)\langle a\rangle} &0\\
			0& \frac{\sin (u)\sqrt{\langle a-1\rangle\langle a+1\rangle}}{\sin (\lambda-u)\langle a\rangle} &1+\frac{\sin (u)\langle a-1\rangle}{\sin(\lambda-u)\langle a\rangle}&0\\
			0&0&0&*
		\end{pmatrix}
	\end{align*}
	
	if $a=1$,
	\begin{align*}
		T_{\Gamma_A}(a):=
		\begin{pmatrix}
			0 &0   &0   &0\\
			0 &\frac{\langle a+1\rangle}{\langle a\rangle}& *&0\\
			0& * & *&0\\
			0&0&0&0
		\end{pmatrix}
	\end{align*}

For $a=L,L-1$, the situation is similar.
\subsection{Transfer matrix and hamiltonian of spin chain}\label{sub:transfer_matrix}
For complicity, we briefly recall the definition of convolution algebra with coefficients in $\pi$ graded algebras over a field and the partial traces to describe the transfer matrix in terms of $\pi$ graded vector space developed in \cite{Felder2020}.

\begin{Def}
Let $\pi$ be a groupoid , a $\pi$ graded algebra $R$ over $k$ is a collection $(R_{\gamma})_{\gamma\in \pi}$ of $k$-vector spaces labeled by arrows of $\pi$ with bilinear products $R_{\alpha}\times R_{\beta}\to R_{\beta\circ \alpha}, (x,y)\to xy$ defined for composable arrows $\alpha,\beta$ and units $1_a\in R_{a}$, for $a\in A$ such that $(i) (xy)z=x(yz)$ whenever defined and $(ii) x1_b=x=1_ax$ for all $x\in R_{\alpha}$ of degree $\alpha\in \pi(a,b)$.
\end{Def}

\begin{Ex}
Let $V\in \text{Vect}_k(\pi)$ and let $\underline{\End}V$ be the $\pi$ graded vector space with $(\underline{\End}V)_{\alpha}=\oplus_{\gamma\in \pi(a,a)}\Hom_k(V_{\alpha\circ \gamma\circ \alpha^{-1}},V_{\gamma})$ where $a=s(a)$, then $\underline{\End} V$ with the product given by the composition of linear maps
\begin{align*}
\Hom_k(V_{\alpha\gamma\alpha^{-1}})\otimes \Hom_{k}(V_{\beta\alpha\gamma\alpha^{-1}\beta^{-1}},V_{\alpha\gamma\alpha^{-1}})\to \Hom_{k}(V_{\beta\alpha\gamma(\beta\alpha)^{-1}},V_{\gamma})
\end{align*}
and unit $1_{a}=\oplus_{\gamma\in \pi(a,a)}\id_{V_{\gamma}}$ is a $\pi$ graded algebra.
\end{Ex}

\begin{Def}
Let $R$ be a $\pi$ graded algebra, the convolution algebra $\Gamma(\pi,R)$ with coefficients in $R$ is the $k$ algebra of maps $f:\pi\to \sqcup_{\alpha\in \pi}R_{\alpha}$ such that
\begin{enumerate}
	\item $f(\alpha)\in R_{\alpha}$ for all arrows $\alpha\in \pi$,\\
	\item for every $a\in A$, there are finitely many $\alpha\in s^{-1}(a)\cup t^{-1}(a)$ such that $f(\alpha)\ne 0$.
\end{enumerate}

The product is the convolution product
\begin{align}
f*g(\gamma)=\sum_{\beta\circ\alpha=\gamma}f(\alpha)g(\beta)
\end{align}
\end{Def}

The partial trace over $V$ is the map
\begin{align}
\tr_{V}:\Hom_{\text{Vect}_k(\pi)}(V\otimes W,W\otimes V)\to \Gamma(\pi,\underline{\End} W)
\end{align}
defined as follows. For $f\in \Hom(V\otimes W,W\otimes V)$ and $\alpha\in \pi(a,b),\gamma\in \pi(a,a)$, let $f(\alpha,\gamma)$ be the component of the mapping, which is the mapping between two paths in the figure \ref{eq:component}.
\begin{align}
f(\alpha,\gamma):V_{\alpha}\otimes W_{\alpha\gamma\alpha^{-1}}\to W_{\gamma}\otimes V_{\alpha}
\end{align}
Define
\begin{align}
\tr_{V_{\alpha}}f(\alpha,\gamma)=\sum_{i}(\id\otimes e^{*}_i)f(\alpha,\gamma)(e_i\otimes \id)\in \Hom(W_{\alpha\gamma\alpha^{-1}},W_{\gamma})
\end{align}
for any basis $e_i$ of $V_{\alpha}$ and dual basis $e^{*}_i$ of the dual vector space $(V_{\alpha})^{*}$.
\begin{equation}\label{eq:component}
f(\alpha,\gamma):=
\begin{tikzcd}
a\arrow[r,"W_{\gamma}"]\arrow[d,"V_{\alpha}"']& a\arrow[d,"V_{\alpha}"]\\
b\arrow[r,"W_{\alpha\gamma\alpha^{-1}}"']&b
\end{tikzcd}
\quad
\tr_{V_{\alpha}}f(\alpha,\gamma)=
\begin{tikzcd}
	a\arrow[r,"W_{\gamma}"]\arrow[d,"\alpha"']& a\arrow[d,"\alpha"]\\
	b\arrow[r,"W_{\alpha\gamma\alpha^{-1}}"']&b
\end{tikzcd}
\end{equation}

\begin{Def}
The partial trace $\tr_{V}f\in \Gamma(\pi,\underline{\End}W)$ of $f\in \Hom_{\text{Vect}_k(\pi)}(V\otimes W,W\otimes V)$ over $V$ is the section
\end{Def}

\begin{align}
\tr_{V}f:\alpha\to \oplus_{\gamma\in \pi(a,a)}\tr_{V_{\alpha}}f(\alpha,\gamma)\in (\underline{\End}W)_{\alpha}
\end{align}

For the vector spaces $V_{i},V_{0}\in \text{Vect}_k(\pi)$ with the dynamical Yang-Baxter operator $\check{R}_{V_0,V_i}(x,a)$, the component $\alpha$ of face type transfer matrix
\[
M(x,\alpha)=\text{tr}_{V_0}\big(\prod_{i=0}^{N-1}\check{R}^{(i+1,i+2)}_{V_0,V_{i+1}}(x, ah^{(i)})\big)(\alpha)\\
\]
can be draw as in the \ref{eq:bigcomponent} which transfer from lower horizontal line to the upper horizontal line, it is an element in $\Gamma(\pi,\underline{\End}(V_1\otimes \dots\otimes V_N))$, its component $\alpha,s(\alpha)=a$ is written as
\begin{equation}\label{eq:bigcomponent}
M(x,\alpha):=
\begin{tikzcd}
a\arrow[r,"V_1"]\arrow[d,"\alpha"']&\bullet\arrow[r,"V_2"]\arrow[d,"V_0"']&\bullet\arrow[r,"V_3"]\arrow[d,"V_0"']&\bullet\arrow[r,"\dots"']\arrow[d,"V_0"']&\bullet\arrow[r,"V_N"]\arrow[d,"V_0"']&a\arrow[d,"\alpha"]\\
b\arrow[r,"V_1"']&\bullet\arrow[r,"V_2"']&\bullet\arrow[r,"V_3"']&\bullet\arrow[r,"\dots"']&\bullet\arrow[r,"V_N"']&b
\end{tikzcd}
\end{equation}

\begin{Prop}By the corallory 3.9 of \cite{Felder2020}, we have the family of  commuting transfer matrix
\begin{equation}
M(x)*M(y)=M(y)*M(x)
\end{equation}
\end{Prop}

Now let $V=V_i\in \text{Vect}_k(\pi), i=0,\dots,N$, suppose that we have the Baxterization with respect to the local Temperley-Lieb operators as in  with the ansatz $\check{R}_{V,V}(x,a)=\id+xT(a)$, the component $(\alpha)$ of transfer matrix can be written as 
\begin{align}
M(x,\alpha)=\text{tr}_{V}\prod_{i=0}^{N-1}\big(\id+xT^{(i+1,i+2)}(ah^{(i)})\big)(\alpha)\\
\end{align}

We can then define the hamiltonian $H(x,\alpha)\in \Gamma(\pi,\underline{\End}(V^{\otimes N}))$ to be the "log derivative" at 0 of the transfer matrix
\begin{equation}\label{eq:Hamitonian}
	H(x,\alpha):=M^{-1}(0,\alpha)M'(x,\alpha)|_{x=0}
	=M^{-1}(0,\alpha)\tr_{V}\big(\sum_{i=0}^{N-1}T^{(i+1,i+2)}(ah^{(i)})\big)
\end{equation}

From the commuting of the transfer matrices, we have the following commuting relations
\begin{Prop}
	\begin{equation}
		H(x)*M(y)=M(y)*H(x)
	\end{equation}
\end{Prop}

We have that $M(0,\alpha)=\tr_{V}(\id)$, if we assume that
\begin{align}\label{eq:assumption}
\dim V_{\alpha}=1,\quad \text{for all $\alpha$}
\end{align}
for simplicity to deal with the partial trace and also enough for all our examples. 

Then we have more explicitly form of $M(0,\alpha)\in \underline{\End}(V^{\otimes N})$
\begin{align}
	M(0,\alpha)|_{V_{\alpha_1}\dots \otimes V_{\alpha_N}}: &V_{\alpha_1}\dots \otimes V_{\alpha_N}\to V_{\alpha_N}\otimes V_1\dots \otimes V_{\alpha_{N-1}}\\
	&v_1\otimes \dots v_N\mapsto v_N\otimes v_1\otimes v_2\dots v_{N-1}
\end{align}
graphically this corresponds to
\begin{equation}
\begin{tikzcd}
	a\arrow[d,"\alpha"']\arrow[r,"\alpha","v_N"']&b\arrow[d,"\alpha_1","v_1"']\arrow[r,"\alpha_1","v_1"']&a_2\arrow[r,"\alpha_2","v_2"']\arrow[d,"\alpha_2"]&\dots\arrow[r,"\alpha_{N-1}","v_{N-1}"']\arrow[d]&a\arrow[d,"\alpha"]\\
	b\arrow[r,"\alpha_1","v_1"']&a_2\arrow[r,"\alpha_2","v_2"']&a_3\arrow[r,"\alpha_3","v_3"']&\dots\arrow[r,"\alpha_N=\alpha","v_N"']&b\\
\end{tikzcd}
\end{equation}

More explicitly, we have different forms
\begin{subequations}\label{eq:Hamiltonian}
\begin{align}
H(x,\alpha)&=M^{-1}(0,\alpha)\big(\sum_{i=1}^{N-2}M(0,\alpha)T^{(i+1,i+2)}(ah^{(i)})+M(0,\alpha)T^{(N,1)}(a)\big)\\
&=\sum_{i=1}^{N-2}T^{(i+1,i+2)}(ah^{(i)})+T^{(N,1)}(a)\\
&=\sum_{i=1}^{N-2}T^{(i+1,i+2)}(ah^{(i)})+M^{-1}(0,\alpha)T^{(1,2)}(a)M(0,\alpha)
\end{align}
\end{subequations}

%
%

\begin{Ex}
	For each representation of local dynamical Temperley-Lieb that can be Baxterized and also satisfies the assumption \eqref{eq:assumption}, we can construct Hamitonians as expressed in \eqref{eq:Hamiltonian}.
\begin{enumerate}	
	\item Plug the representation in Proposition \ref{prop:tri_unres} and \ref{prop:hyp_unres} into \eqref{eq:Hamiltonian}, this corresponds to the case of trigonometric and hyperbolic unrestricted type $A$.\\
	\item Plug the representation in Proposition \ref{prop:res_ade} and \ref{prop:res_affineade} into \eqref{eq:Hamiltonian}, this corresponds to the case of restricted ADE and affine ADE type.
\end{enumerate}
\end{Ex}

\begin{Rem}
The local dynamical Temperley--Lieb structure underlying these one dimensional integrable systems are probably new. As a physics system expressed in the square lattice language, these systems has been widely studied in physics, for example the work of \cite{BAZHANOV1989}, where they considered restricted type A and calculate the eigenvectors and eigenvalues.
\end{Rem}

\appendix

\section{Perron-Frobenius theorem}\label{PF_theorem}
In this appendix, we recall the classical Perron-Frobenius theorem in the form of the book \cite{Etingof2015}, Theorem 3.2.1.

\begin{Th}[Frobenius-Perron] Let $B$ be a square matrix with non-negative real entries.
	
\begin{enumerate}
	\item $B$ has a non-negative real eigenvalue. The largest non-negative real eigenvalue $\lambda(B)$ of $B$ dominates the absolute values of all other eigenvalues $\mu$ of $B:|\mu|\le \lambda(B)$ (in other words, the spectral radius of $B$ is an eigenvalue.) Moreover, there is an eigenvector of $B$ with non-negative entries and eigenvalue $\lambda(B)$.\\
	\item If $B$ has strictly positive entries then $\lambda(B)$ is a simple positive eigenvalue, and the corresponding eigenvector can be normalized to have strictly positive entries. Moreover, $|\mu|<\lambda(B)$ for any other eigenvalue $\mu$ of $B$.\\
	\item If a matrix $B$ with non-negative entries has an eigenvector $v$ with strictly positive entries, then the corresponding eigenvalue is $\lambda(B)$.\\
\end{enumerate}
\end{Th}
\begin{proof}
See the proof of \cite{Etingof2015}, Theorem 3.2.1.
\end{proof}

As a example of the above theorem, we have the following table \ref{PF-V} of Perro-Frobenius eigenvector of classical ADE Dykin diagram, see for example \cite{PEARCE1990}.

\begin{table}[h]
	\begin{center}
		\begin{tabular}{|c|c|}
			\hline
			Lie lagebra & Perro-Frobenius Eigenvector\\
			$A_L$       &$(\sin(\frac{\pi}{L+1}),\sin\frac{2\pi}{L+1},\dots ,\sin \frac{L\pi}{L+1})$\\
			$D_L$      &$(2\cos\frac{(L-2)\pi}{2L-2},\dots,2\cos\frac{2\pi}{2L-2},2\cos\frac{\pi}{2L-2},1,1)$\\
			$E_6$ & $(\sin \frac{\pi}{12},\sin\frac{\pi}{6},\sin\frac{\pi}{4},\sin\frac{\pi}{3}-\frac{\sin\frac{\pi}{4}}{2\cos\frac{\pi}{12}},\sin\frac{5\pi}{12}-\sin\frac{\pi}{4},\frac{\sin \frac{\pi}{4}}{2\cos \frac{\pi}{18}})$\\
			$E_7$ & $(\sin\frac{\pi}{18},\sin\frac{\pi}{9},\sin\frac{\pi}{6},\sin\frac{2\pi}{9},\sin \frac{2\pi}{18}-\frac{\sin \frac{2\pi}{9}}{2\cos\frac{\pi}{18}},\sin\frac{\pi}{3}-\sin\frac{2\pi}{9},\frac{\sin \frac{2\pi}{9}}{2\cos\frac{\pi}{30}})$\\
			$E_8$ & $(\sin\frac{\pi}{30},\sin\frac{\pi}{15},\sin\frac{\pi}{10},\sin\frac{2\pi}{15},\sin\frac{\pi}{6},\sin\frac{\pi}{5}-\frac{\sin\frac{\pi}{6}}{2\cos\frac{\pi}{30}},\sin\frac{7\pi}{30}-\sin\frac{\pi}{6},\frac{\sin\frac{\pi}{6}}{2\cos\frac{\pi}{30}})$\\
			$A^{(1)}_{L-1}$ & $(1,\dots,1)$\\
			$D^{(1)}_{L-1}$ & $(1,1,2,2,\dots,2,2,1,1)$\\
			$E^{(1)}_{6}$   & $(1,2,3,2,1,2,1)$\\
			$E^{(1)}_{7}$   & $(1,2,3,4,3,2,1,2)$\\
			$E^{(1)}_{8}$   & $(1,2,3,4,5,6,4,2,3)$.\\
			\hline
		\end{tabular}
	\end{center}
	\caption{Table of Eigenvector}\label{PF-V}
\end{table}

\bibliographystyle{plain}
\bibliography{dynamical_baxterization}

\end{document}